\documentclass[a4paper]{amsart}
\usepackage[margin=3cm]{geometry}
\usepackage{mathtools}
\usepackage{amsmath}
\usepackage{amsthm}
\usepackage{graphicx}
\usepackage{subfig}
\usepackage{cite}
\usepackage{booktabs}
\usepackage{color}
\usepackage{url}
\definecolor{gray}{RGB}{150,150,150}

\newcommand{\diff}[2]{\frac{\mathrm{d} #1}{\mathrm{d} #2}}
\newcommand{\pfrac}[2]{\frac{\partial #1}{\partial #2}}
\newcommand{\ppfrac}[2]{\frac{\partial^2 #1}{\partial #2^2}}
\newcommand{\eref}[1]{(\ref{#1})}
\newcommand{\bi}[1]{\text{\bfseries\itshape #1\/}}

\newtheorem{theorem}{Theorem}[section]
\renewenvironment{proof}[1][Proof]{\begin{trivlist}
\item[\hskip \labelsep {\bfseries #1}]}{\end{trivlist}}
\renewcommand{\qed}{\nobreak \ifvmode \relax \else
      \ifdim\lastskip<1.5em \hskip-\lastskip
      \hfill \fi \nobreak
      \vrule height0.5em width0.5em depth0em\fi}
\newtheorem{lemma}[theorem]{Lemma}
\theoremstyle{remark}
\newtheorem{remark}{Remark}[section]
\newtheorem{corollary}[theorem]{Corollary}
\newtheorem{conjecture}[theorem]{Conjecture}

\begin{document}

\title{Novel solutions for a model of wound healing angiogenesis}
\author[K.~Harley \and P.~van Heijster \and R.~Marangell \and G.J.~Pettet \and M.~Wechselberger]{K.~Harley$^{\ast}$ \and P.~van Heijster$^{\ast}$ \and R.~Marangell$^{\dagger}$ \and G.J.~Pettet$^{\ast}$ \and M.~Wechselberger$^{\dagger}$}

\thanks{$^{\ast}$Mathematical Sciences School, Queensland University of Technology, Brisbane, QLD 4000 Australia}
\thanks{$^{\dagger}$School of Mathematics and Statistics, University of Sydney, Sydney, NSW 2006 Australia}

\begin{abstract}
We prove the existence of novel, shock-fronted travelling wave solutions to a model of wound healing angiogenesis studied in Pettet {\em et al}., IMA J.~Math.~App.~Med., {\bf 17}, 2000.  In this work, the authors showed that for certain parameter values, a heteroclinic orbit in the phase plane representing a smooth travelling wave solution exists.  However, upon varying one of the parameters, the heteroclinic orbit was destroyed, or rather {\em cut-off}, by a {\em wall of singularities} in the phase plane.  As a result, they concluded that under this parameter regime no travelling wave solutions existed.  Using techniques from geometric singular perturbation theory and canard theory, we show that a travelling wave solution actually still exists for this parameter regime: we construct a heteroclinic orbit passing through the wall of singularities via a folded saddle canard point onto a repelling slow manifold.  The orbit leaves this manifold via the fast dynamics and lands on the attracting slow manifold, finally connecting to its end state. This new travelling wave is no longer smooth but exhibits a sharp front or shock.  Finally, we identify regions in parameter space where we expect that similar solutions exist.  Moreover, we discuss the possibility of more exotic solutions.  
\end{abstract}

\maketitle

\section{Introduction}
\label{sec:intro}

\subsection{The model}
\label{subsec:model}

We study a two species model developed in \cite{Pettet_96_phd,Pettet_McElwain_Norbury_00} describing wound healing angiogenesis.  This model focuses on the migration of microvessel endothelial cells (MEC), especially those that make up the tips of newly formed capillaries, into the wound space, mediated by the presence of a chemoattractant: macrophage derived growth factor (MDGF).   The interaction between these two species is modelled using Lotka--Volterra like, predator-prey interactions, with the capillary tips (MEC) acting as the predator and MDGF acting as the prey.  An additional chemotaxis term describes the capillary tip migration in response to a gradient of MDGF.  Due to the assumed symmetry of the wound, the model can be restricted to a one-dimensional spatial domain; see the left hand panel of Figure~\ref{fig:wound}.  The model as described in \cite{Pettet_McElwain_Norbury_00} is
\begin{equation}\label{eq:unscaled-model}\begin{aligned}
\pfrac{a}{\hat{t}} &= \lambda_1 a \left(1 - \frac{a}{K}\right) - \lambda_2 an, \\
\pfrac{n}{\hat{t}} &= -\chi\pfrac{}{\hat{x}}\left(n\pfrac{a}{\hat{x}}\right) + \lambda_3 an - \lambda_4 n, 
\end{aligned}\end{equation}
with $\hat{x} \in [0,L]$, $\hat{t} > 0$, $K, \chi > 0$, $\lambda_i > 0$ for $i = 1,\ldots,4$, and boundary conditions given by
\[ n(0,\hat{t}) = \lambda_1\left(\frac{\lambda_3 K - \lambda_4}{\lambda_2\lambda_3 K}\right), \quad a(0,\hat{t}) = \frac{\lambda_4}{\lambda_3}, \quad a_{\hat{x}}(0,\hat{t}) = 0, \] 
and
\[ n(L,\hat{t}) = 0, \quad a(L,\hat{t}) = K, \quad a_{\hat{x}}(L,\hat{t}) = 0, \]
where the subscript denotes the partial derivative.  Here $a(\hat{x},\hat{t})$ represents the concentration of the chemoattractant MDGF and $n(\hat{x},\hat{t})$ the capillary tip density.  Moreover, $\hat{x} = 0$ corresponds to the edge of the wound and $\hat{x} = L$ to the centre; see Figure~\ref{fig:wound}.  The first term in the expression for $a_{\hat{t}}$ describes the production of MDGF by the body in response to the wounding, with $K$ the carrying capacity of MDGF within the wound.  The second term describes the consumption of MDGF by the MEC at the capillary tips.  The advection term in the expression for $n_{\hat{t}}$ describes the migration of the capillary tips up the gradient of MDGF due to chemotaxis.  The kinetic terms describe the birth of MEC at the capillary tip due to the presence of MGDF, and natural cell death, respectively.  We refer to \cite{Pettet_McElwain_Norbury_00} for a more detailed description and derivation of the model.  

\begin{figure}[ht]
\centering
\includegraphics[width=0.42\textwidth]{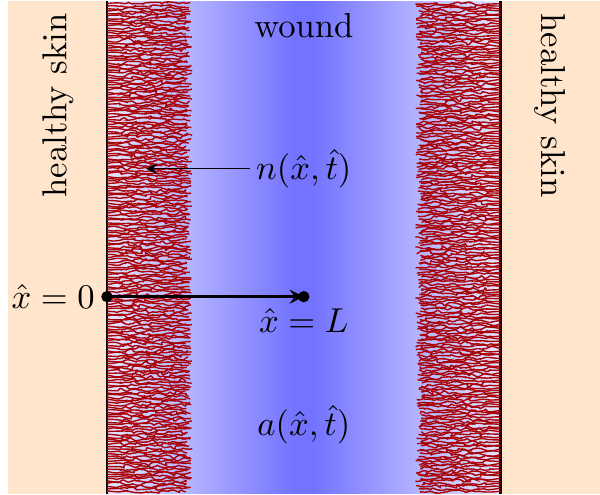}
\includegraphics[width=0.57\textwidth]{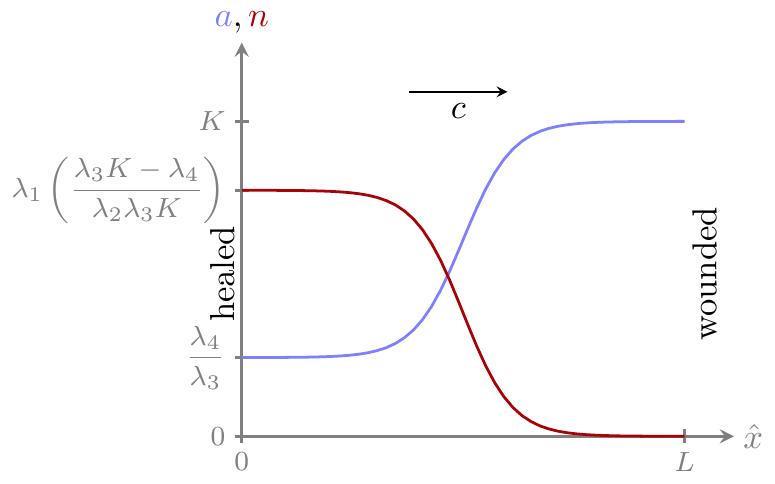}
\caption{A schematic of a wound and the profiles of the density of capillary tips $n(\hat{x},\hat{t})$ and chemoattractant MDGF $a(\hat{x},\hat{t})$ within it.}
\label{fig:wound}
\end{figure}

The goal is to find travelling wave solutions that connect the wounded steady state $(a_W,n_W) = (K,0)$ to the healed steady state $(a_H,n_H) = \left(\dfrac{\lambda_4}{\lambda_3},\lambda_1\left(\dfrac{\lambda_3 K - \lambda_4}{\lambda_2\lambda_3 K}\right)\right)$.  To do so, the domain is extended to $\hat{x} \in \mathbb{R}$.  Biologically, this means that any travelling wave solutions we find describe the closing of the wound for early times in the healing process, when the interaction with the corresponding wave from the other edge of the wound is negligible.  Further analysis is required to investigate the {\em filling} of the wound space as the edges come together.  

Working on the unbounded domain, we nondimensionalise \eref{eq:unscaled-model} via
\[ u = \frac{a}{K}, \quad w = \frac{\lambda_2 n}{\lambda_1}, \quad x = \sqrt{\frac{\lambda_1}{\chi K}}\hat{x}, \quad t = \lambda_1\hat{t}, \quad \alpha = \frac{\lambda_4}{\lambda_1}, \quad \beta = \frac{\lambda_3 K}{\lambda_4}, \]
to give
\begin{equation}\label{eq:model}\begin{aligned}
\pfrac{u}{t} &= u \left(1 - u - w \right), \\
\pfrac{w}{t} &= -\pfrac{}{x}\left(w\pfrac{u}{x}\right) + \alpha w(\beta u - 1),
\end{aligned}\end{equation}
with $x \in \mathbb{R}$, $t > 0$ and $\alpha, \beta > 0$.  We remark that our choice of nondimensionalisation differs from the one used in \cite{Pettet_McElwain_Norbury_00}.  However, we can relate our parameters to theirs: $\alpha \leftrightarrow \hat{\lambda}_4$ and $\beta \leftrightarrow \hat{\lambda}_3/\hat{\lambda}_4$, where we draw attention to the fact that $\hat{\lambda}_3$ and $\hat{\lambda}_4$ refer to the scaled parameters post nondimensionalisation in \cite{Pettet_McElwain_Norbury_00}, not the original $\lambda_3$ and $\lambda_4$ in \eref{eq:unscaled-model}.  Our particular choice of nondimensionalisation is motivated by the fact that the rescaled background states of \eref{eq:model} of interest to us are 
\[ (u_W,w_W) = (1,0) \quad {\rm and} \quad (u_H,w_H) = \left(\frac{1}{\beta},1 - \frac{1}{\beta}\right); \]
that is, the wounded state is independent of the model parameters and the healed state only depends on $\beta$.  From now on we only consider $\beta > 1$ to ensure that $(u_H,w_H)$ lies in the positive quadrant.  

\subsection{Previous results}
\label{subsec:previous-results}

In \cite{Pettet_McElwain_Norbury_00}, the authors investigate travelling wave solutions to \eref{eq:model} by looking for heteroclinic orbits in the phase plane of the system obtained by substituting the first expression of \eref{eq:model} into the advection term of the second, after transforming to the comoving frame $z = x - ct$.  In our scaling, this system is
\begin{equation}\label{eq:odes-withf}\begin{aligned}
\diff{u}{z} &= -\frac{u f(u,w)}{c}, \\
\left(c^2 + u f(u,2w)\right)\diff{w}{z} &= \frac{uw f(u,w)f(2u,w)}{c} - \alpha cw(\beta u - 1), 
\end{aligned}\end{equation}
where for notational convenience we have introduced
\begin{equation}\label{eq:f}
f(u,w) = 1 - u - w.  
\end{equation}

The phase plane analysis of \eref{eq:odes-withf} is complicated by the term premultiplying the $w$-derivative.  When this term vanishes the system becomes singular; this curve is referred to as the {\em wall of singularities} in \cite{Pettet_96_phd,Pettet_McElwain_Norbury_00}.  The wall of singularities for \eref{eq:odes-withf} is given by
\begin{equation}\label{eq:wall}
w = \frac{c^2 + u - u^2}{2u} \eqqcolon F(u)
\end{equation}
and will be represented by the green dotted line in the forthcoming figures.  Phase trajectories cannot cross the wall of singularities except at points where the right hand side of the ODE for $w$ also vanishes (that is, the ODEs are no longer singular).  These points are referred to as {\em gates} or {\em holes} in the wall of singularities \cite{Pettet_96_phd,Pettet_McElwain_Norbury_00}.  The $u$-locations of the holes in the wall are given by the roots of
\begin{equation}\label{eq:poly}
3u^4 - 4u^3 + \left[1 + 4c^2(1 - \alpha\beta)\right]u^2 + 2c^2(2\alpha - 1)u + c^4 = 0, 
\end{equation}
which is obtained by equating both the left and right hand sides of the second equation in \eref{eq:odes-withf} to zero, assuming $w \neq 0$.  

Upon constructing phase planes of \eref{eq:odes-withf}, the authors of \cite{Pettet_McElwain_Norbury_00} found that under certain parameter regimes a smooth heteroclinic orbit connecting $(u_H,w_H)$ and $(u_W,w_W)$ could be identified, while under other parameter regimes, no such orbit could be identified due to interference from the wall of singularities.  To demonstrate this result, \cite{Pettet_McElwain_Norbury_00} provides phase planes for two parameter regimes; in our scaling these correspond to
\begin{equation}\label{eq:cases}
\textrm{Case 1:} \quad (\alpha, \beta, c) = \left(\dfrac{2}{5},\dfrac{5}{2},1\right), \qquad \textrm{Case 2:} \quad (\alpha, \beta, c) = \left(\dfrac{2}{5},\dfrac{5}{2},\dfrac{\sqrt{2}}{2}\right). 
\end{equation}
Schematics of the phase planes provided in \cite{Pettet_McElwain_Norbury_00} for the two parameter regimes are given in Figure~\ref{fig:cut-off-schematic}.  

\begin{figure}[ht]
\centering
\subfloat[Case 1]{\label{fig:cut-off-schematic-smooth}\includegraphics{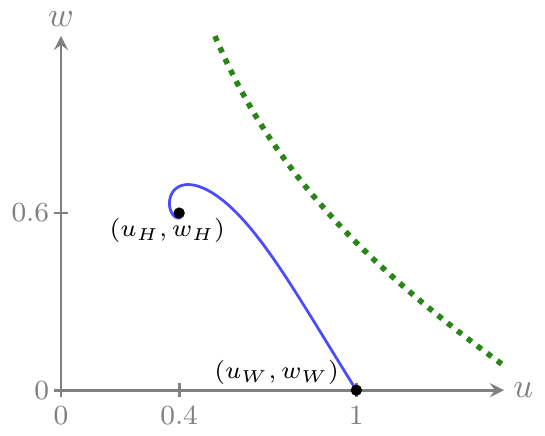}}
\subfloat[Case 2]{\label{fig:cut-off-schematic-shock}\includegraphics{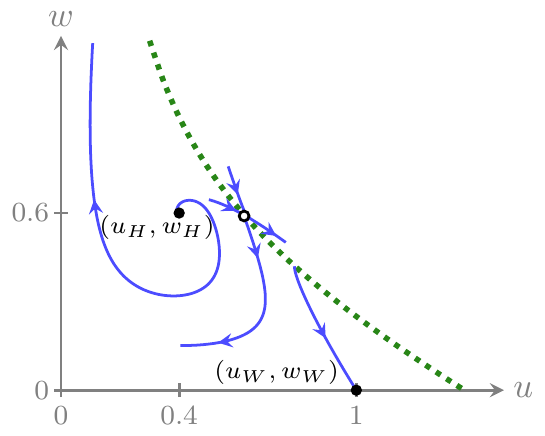}}
\caption{Schematics of figures in \cite{Pettet_McElwain_Norbury_00} illustrating how the smooth heteroclinic connection is {\em cut-off} by the wall of singularities (green dotted line) as $c$ is decreased.  The gate in the right hand panel is depicted by the open circle, with a trajectory passing through it acting as a separatrix between the end states.}
\label{fig:cut-off-schematic}
\end{figure}

More specifically, in the case illustrated in Figure~\ref{fig:cut-off-schematic-smooth}, a smooth heteroclinic connection could be identified as the wall of singularities is sufficiently distant from the end states $(u_H,w_H)$ and $(u_W,w_W)$.  Furthermore, since there are no gates in the wall, that is, \eref{eq:poly} has no real, positive solutions, it was concluded that the connection found numerically, exists.  Alternatively, in the case illustrated in Figure~\ref{fig:cut-off-schematic-shock}, reducing $c$ causes the wall of singularities to move closer to the end states, and as a result, it \emph{cuts off} the smooth trajectory that existed previously.  Decreasing $c$ also causes a gate in the wall of singularities to appear, as illustrated in Figure~\ref{fig:cut-off-schematic-shock}.  However, one of the trajectories leaving the gate acts as a separatrix between $(u_H,w_H)$ and $(u_W,w_W)$ and so it was concluded that a heteroclinic connection did not exist.  

It was also postulated that if $c$ was decreased further, such that the wall of singularities lies between $(u_H,w_H)$ and $(u_W,w_W)$, no travelling wave solutions could exist as the wall separating the end states precludes a heteroclinic connection.  

Ultimately, the question remained: under what parameter regimes did heteroclinic orbits, and correspondingly travelling wave solutions, exist or not exist?  

Subsequent work on similar systems with walls of singularities in the phase plane suggests that sometimes heteroclinic connections can still be made via interactions with the wall, leading to shock-fronted travelling wave solutions \cite{Marchant_Norbury_Perumpanani_00, Landman_Pettet_Newgreen_03, Landman_Simpson_Slater_Newgreen_05}.  Furthermore, in \cite{Marchant_Norbury_02}, the authors investigate specific solutions of \eref{eq:model} that have semi-compact support in $w$.  By combining phase plane analysis with the Rankine--Hugoniot and Lax entropy conditions for shock solutions from hyperbolic PDE theory, they numerically identify two waves of this kind; one for $\alpha = 1/5$, $\beta = 5$ and $c \approx 0.72$, and another for $\alpha = 7/10$, $\beta = 10/7$ and $c \approx 0.24$.  They also make claims about the existence (and non-existence) of travelling wave solutions in other parameter regimes, but without providing details.  

Thus, we refine the outstanding question of \cite{Pettet_McElwain_Norbury_00}: do heteroclinic orbits still exist for values of $c$ where the wall of singularities cuts off the smooth connection?  Furthermore, do heteroclinic orbits exist under other parameter regimes that have not been previously considered?  And finally, can we determine which parameter regimes support smooth travelling wave solutions, which support shock-fronted travelling wave solutions, or which do not support travelling wave solutions?  Moreover, we ask these questions for a more general model than \eref{eq:model} that includes a small amount of diffusion of both species, which is biologically more relevant.  

\subsection{Main results and outline}
\label{subsec:main-result}

In the original development of the model, diffusion of both MDGF and the capillary tip cells was neglected as it was assumed that the kinetic and advective terms played a significantly larger role in the distribution of MDGF and MEC than diffusion \cite{Pettet_McElwain_Norbury_00}.  In the current article, we do not neglect diffusion but rather assume it to be small.  Consequently, the system we consider is
\begin{equation}\label{eq:model-with-diff}\begin{aligned}
\pfrac{u}{t} &= u f(u,w) + \varepsilon\ppfrac{u}{x}, \\
\pfrac{w}{t} &= -\pfrac{}{x}\left(w\pfrac{u}{x}\right) + \alpha w(\beta u - 1) + \varepsilon\ppfrac{w}{x}, 
\end{aligned}\end{equation}
with $0 \leq \varepsilon \ll 1$, $u,w \geq 0$, $x \in \mathbb{R}$, $t > 0$, $\alpha > 0$, $\beta > 1$ and $f(u,w)$ defined in \eref{eq:f}.  We remark that, as in \cite{Wechselberger_Pettet_10, Harley_vanHeijster_Marangell_Pettet_Wechselberger_13}, while for convenience we take the diffusivities to be equal, the results would not be significantly altered if we chose diffusivities that are not equal but still of the same asymptotic order.  This is demonstrated in \cite{confproc}, where the ratio of the diffusivities is taken to be an $\mathcal{O}(1)$ parameter.  The aim is to find travelling wave solutions connecting the healed state to the wounded state:
\[ \lim_{x \to -\infty}{u(x)} = \frac{1}{\beta}, \quad \lim_{x \to \infty}{u(x)} = 1, \quad \lim_{x \to -\infty}{w(x)} = 1 - \frac{1}{\beta}, \quad \lim_{x \to \infty}{w(x)} = 0.  \] 

Including a small amount of diffusion in the model not only means that the model becomes biologically more realistic but that mathematically we are dealing with a singularly perturbed system rather than purely hyperbolic PDEs.  Consequently, \eref{eq:model-with-diff} is amenable to analysis using techniques from geometric singular perturbation theory (GSPT) \cite{Jones_95, Kaper_99} and canard theory \cite{Benoit_Callot_Diener_Diener_81, Wechselberger_12}, following the method outlined in \cite{Wechselberger_Pettet_10}.  

The advantage of this approach lies in the ability to add rigor to formal asymptotic results of standard singular perturbation methods or numerical results, such as used in \cite{Marchant_Norbury_02}.  Embedding the problem, through the inclusion of small diffusion, into a higher dimensional (phase-)space allows us to identify a slow (invariant) manifold along which the solutions evolve, in the slow scaling.  Furthermore, recognising the equivalence of holes in the wall of singularities and canard points \cite{Wechselberger_Pettet_10} provides us with a clear interpretation of the solution behaviour near such points.  In the fast scaling, the fast {\em jumps} automatically encode the Rankine--Hugoniot and Lax entropy conditions for shocks from classical hyperbolic PDE theory.  

As in \cite{Wechselberger_Pettet_10}, we write \eref{eq:model-with-diff} as a system of coupled balance laws by introducing a dummy variable $v = u_x$: 
\[ \begin{pmatrix} u \\ v \\ w \end{pmatrix}_t + \begin{pmatrix} 0 \\ -u f(u,w) \\ vw \end{pmatrix}_x = \begin{pmatrix} u f(u,w) \\ 0 \\ \alpha w(\beta u - 1) \end{pmatrix} + \varepsilon \begin{pmatrix} u \\ v \\ w \end{pmatrix}_{xx}. \]
Since we are looking for travelling wave solutions, we are interested in solutions of
\[ \begin{aligned}
\left(\varepsilon u_z + cu\right)_z &= -u f(u,w), \\
\left(\varepsilon v_z + cv + u f(u,w)\right)_z &= 0, \\
\left(\varepsilon w_z + cw - vw\right)_z &= -\alpha w(\beta u - 1), 
\end{aligned} \]
with $z = x - ct$ the travelling wave coordinate, as before.  We look for right-moving travelling waves (see Figure~\ref{fig:wound}) and therefore assume $c > 0$.  The above can be written as a six-dimensional system of first order ODEs, via the introduction of three, new, slow variables
\[ \begin{aligned}
\hat{u} &\coloneqq \varepsilon u_z + cu, \\
\hat{v} &\coloneqq \varepsilon v_z + cv + u f(u,w), \\
\hat{w} &\coloneqq \varepsilon w_z + cw - vw,  
\end{aligned} \]
to give
\begin{equation}\label{eq:slow}\begin{aligned}
\hat{u}_z &= -u f(u,w), \\
\hat{v}_z &= 0, \\
\hat{w}_z &= -\alpha w(\beta u - 1), \\
\varepsilon u_z &= \hat{u} - cu, \\
\varepsilon v_z &= \hat{v} - cv - u f(u,w), \\
\varepsilon w_z &= \hat{w} - cw + vw. 
\end{aligned}\end{equation}
We refer to this system as the {\em slow system}, with $z$ the {\em slow travelling wave coordinate}.  The differential equation for $\hat{v}$ implies $\hat{v}$ is a constant.  A straightforward computation shows that $\hat{v} = 0$ and hence in principle we have only a five-dimensional system of equations.  For $\varepsilon \neq 0$, we obtain the equivalent {\em fast system} by introducing the {\em fast travelling wave coordinate} $y = z/\varepsilon$:
\begin{equation}\label{eq:fast}\begin{aligned}
\hat{u}_y &= -\varepsilon u f(u,w), \\
\hat{w}_y &= -\varepsilon \alpha w(\beta u - 1), \\
u_y &= \hat{u} - cu, \\
v_y &= -cv - u f(u,w), \\
w_y &= \hat{w} - cw + vw, 
\end{aligned}\end{equation}
where we have removed $\hat{v}$ from the system.  

In the singular limit $\varepsilon \to 0$, the slow and fast systems \eref{eq:slow} and \eref{eq:fast} reduce, respectively, to what are termed the {\em reduced problem},
\begin{equation}\label{eq:reduced}\begin{aligned}
\hat{u}_z &= -u f(u,w), \\
\hat{w}_z &= -\alpha w(\beta u - 1), \\
0 &= \hat{u} - cu, \\
0 &= -cv - u f(u,w), \\
0 &= \hat{w} - cw + vw, 
\end{aligned}\end{equation}
and the {\em layer problem},
\begin{equation}\label{eq:layer}\begin{aligned}
\hat{u}_y &= 0, \\
\hat{w}_y &= 0, \\
u_y &= \hat{u} - cu, \\
v_y &= -cv - u f(u,w), \\
w_y &= \hat{w} - cw + vw. 
\end{aligned}\end{equation}
Note that the singular limit problems are no longer equivalent.  

The strategy is now as follows.  The two singular limit systems are analysed independently as, being lower dimensional, they are more amenable to analysis than the full system.  Then, using the results from these, we construct singular limit solutions by concatenating components from each of the subsystems, in the appropriate spatial domain.  These concatenations provide us with singular heteroclinic orbits, connecting $(u_H,w_H)$ to $(u_W,w_W)$.  Finally, GSPT and canard theory allows us to prove, under certain conditions, that these singular heteroclinic orbits persist as nearby orbits of the full system for $0 < \varepsilon \ll 1$, and correspondingly, that a travelling wave solution exists.  

We implement this strategy to prove the existence of travelling wave solutions to \eref{eq:model-with-diff}.  Initially, this is done in the general case, that is, without specifying $\alpha$, $\beta$ or $c$.  However, it is algebraically too involved to derive specific results without eventually specifying the parameters.  This is due the variability in the number and type of canard points (or holes in the wall of singularities) corresponding to the roots of \eref{eq:poly}, as well as the location and classification (by stability) of the healed state, as parameters are varied.  Therefore, the purpose of this article is two-fold.  Firstly, to obtain rigorous results, we focus on the two choices of parameter values used in \cite{Pettet_McElwain_Norbury_00}, given in \eref{eq:cases}.  This leads to our main result:

\begin{theorem}\label{theorem:outline}
Under certain mild assumptions and for sufficiently small $0 \leq \varepsilon \ll 1$, \eref{eq:model-with-diff} possesses travelling wave solutions, connecting $(u_H,w_H)$ to $(u_W,w_W)$.  In particular, in Case 1, a smooth travelling wave solution exists, while in Case 2, a travelling wave solution containing a shock (in the singular limit $\varepsilon \to 0$) exists.  
\end{theorem}

We prove Theorem~\ref{theorem:outline} in Section~\ref{sec:gspm}.  The proof is similar to the one in \cite{Harley_vanHeijster_Marangell_Pettet_Wechselberger_13}, which follows the outline provided by \cite{Wechselberger_Pettet_10}.  As a result, we omit some of the more technical details and instead refer the reader to the previous works and the appendices.  

While Theorem~\ref{theorem:outline} only refers to two specific parameter sets, the results and methods of Section~\ref{sec:gspm} apply more generally.  Therefore, the secondary purpose of this article is to infer for more general parameter values, the types of travelling wave solutions that may be observed.  In Section~\ref{sec:general-results}, we identify regions in parameter space where we expect to observe qualitatively similar results to those of Theorem~\ref{theorem:outline}.  Moreover, we discuss the existence of other possible travelling wave solutions of \eref{eq:model-with-diff} for other parameter regimes.  These other solutions depend on the location of the equilibrium point $(u_H,w_H)$ relative to the wall of singularities, its stability, and the number and type of gates or canard points in the positive quadrant.  

We conclude with a brief discussion of the significance of our results from a biological perspective.  

\section{Geometric singular perturbation methods}
\label{sec:gspm}

In this section, we follow \cite{Wechselberger_Pettet_10,Harley_vanHeijster_Marangell_Pettet_Wechselberger_13} to construct travelling wave solutions to \eref{eq:model-with-diff} using techniques from GPST and canard theory.  As previously discussed, this is done in the first instance by analysing the two singular limit systems \eref{eq:reduced} and \eref{eq:layer} independently, in the appropriate spatial regions.  We then use the information gathered from each of the systems to prove Theorem~\ref{theorem:outline} and construct travelling wave solutions to \eref{eq:model-with-diff} for sufficiently small $0 \leq \varepsilon \ll 1$.  

\subsection{Layer problem}
\label{subsec:layer}

We begin the analysis with the layer problem \eref{eq:layer} and note that the layer problem is independent of the kinetic parameters $\alpha$ and $\beta$.  The results for the layer problem will therefore hold for any $\alpha$ and $\beta$.  The steady states of \eref{eq:layer} define a critical manifold and are given by the surface
\begin{equation}\label{eq:S}
S = \left\{ (u,v,w,\hat{u},\hat{w}) \; \middle| \; \hat{u} = cu, v = -\frac{u f(u,w)}{c}, \hat{w} = cw - vw \right\},
\end{equation}
where we recall that $f(u,w)$ is given by \eref{eq:f}.  

\begin{lemma}\label{lemma:S}
The critical manifold $S$ is folded with one attracting side and one repelling side.  Moreover, the fold curve $F$, projected into $(u,w)$-space, coincides with the wall of singularities defined by \eref{eq:wall}.  
\end{lemma}

\begin{proof}
The proof of Lemma~\ref{lemma:S} is similar to the proof of Lemma~2.2 in \cite{Harley_vanHeijster_Marangell_Pettet_Wechselberger_13} and hence we refer to Appendix~\ref{ap:layer} for the details.  
\qed\end{proof}

The critical manifold can be written $S = S_r \cup F \cup S_a$, where $S_r$ corresponds to the repelling component of $S$, $S_a$ to the attracting component and $F$ to the fold curve.  Figure~\ref{fig:S} gives an illustration of $S$ for the parameter regimes in \eref{eq:cases}, projected into $(u,w,\hat{w})$-space.  

\begin{figure}[ht]
\centering
\subfloat[Case 1]{\includegraphics[width=0.5\textwidth]{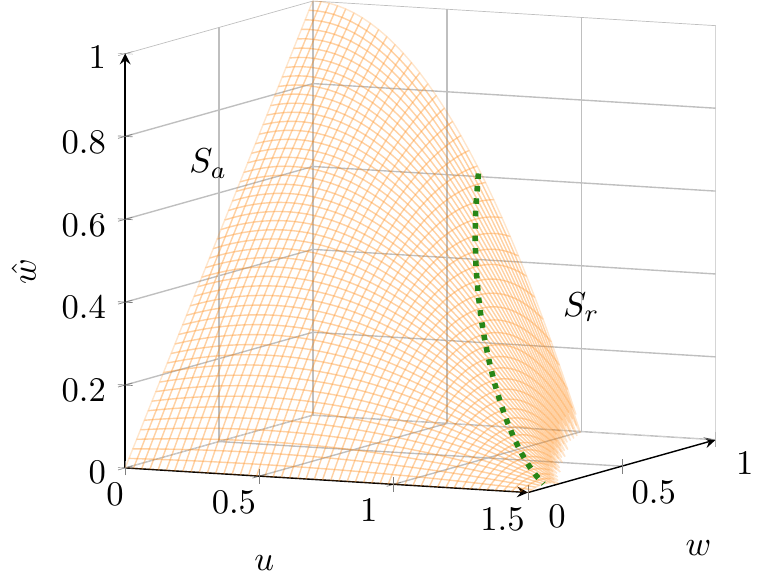}}
\subfloat[Case 2]{\includegraphics[width=0.5\textwidth]{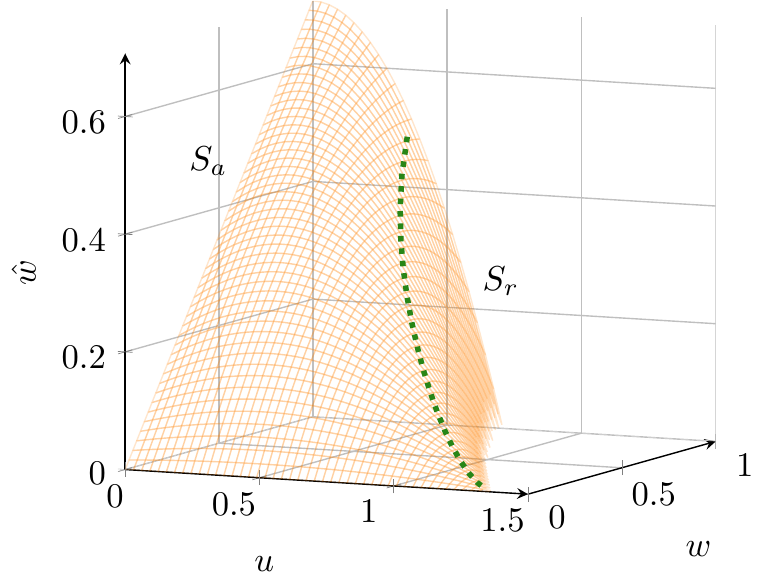}}\\
\subfloat[Case 1]{\includegraphics[width=0.5\textwidth]{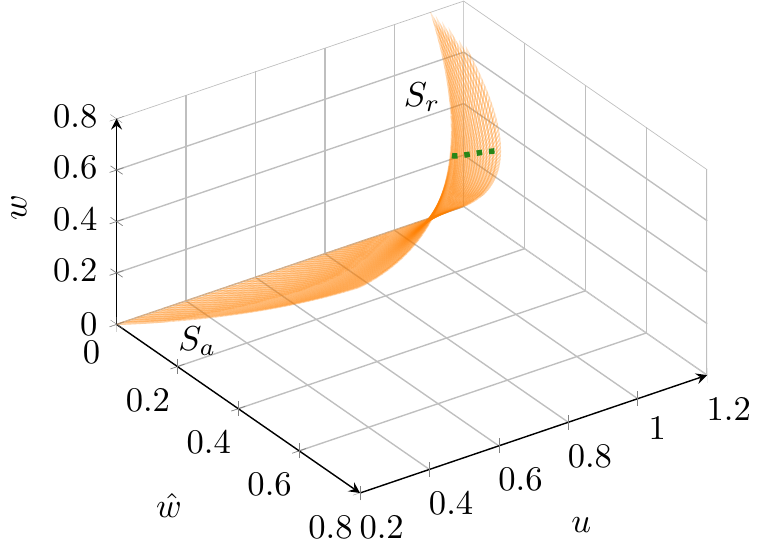}}
\subfloat[Case 2]{\includegraphics[width=0.5\textwidth]{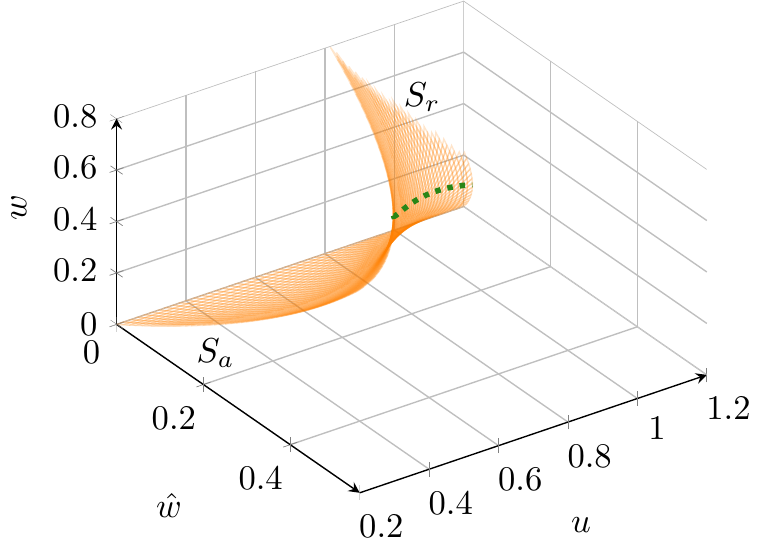}}\\
\caption{The critical manifold $S$ defined in \eref{eq:S} projected into $(u,w,\hat{w})$-space, for Case 1 and Case 2.  The green dotted line corresponds to the fold curve $F$.}
\label{fig:S}
\end{figure}

Since $\hat{u}$ and $\hat{w}$ act as parameters of the layer problem ($\hat{u}_y = \hat{w}_y = 0$), the layer flow connects points on $S_r$ to points on $S_a$ with constant $\hat{u}$ and $\hat{w}$.  Also, $u$ must be constant along any trajectory within the layer problem in order to satisfy both the third equation in \eref{eq:layer} and the condition that $u = \hat{u}/c$ on $S$.  A schematic is given in Figure~\ref{fig:S-schematic}.  

\begin{figure}[ht]
\centering
\includegraphics{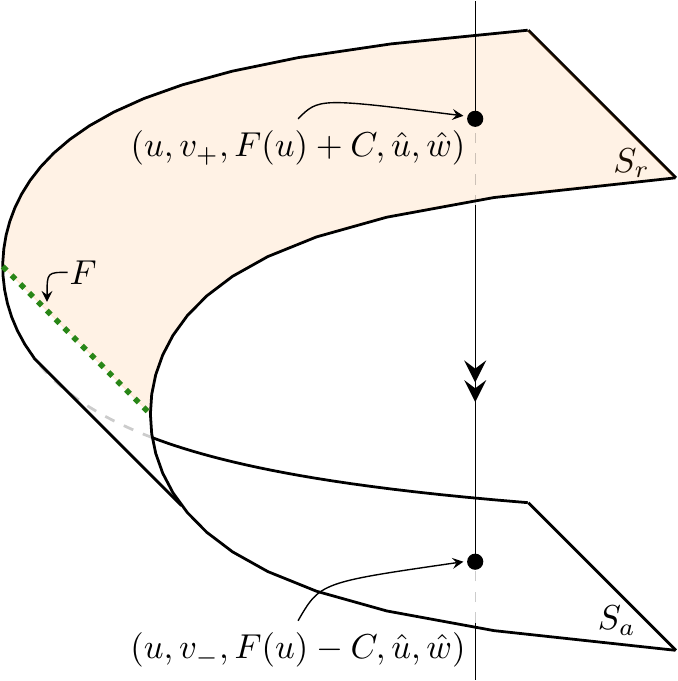}
\caption{A schematic of the critical manifold $S$ and the flow through the layer problem.}
\label{fig:S-schematic}
\end{figure}

\begin{remark}
It can be shown that the condition that $\hat{u}$ and $\hat{w}$ are constant along trajectories of the layer problem is equivalent to shocks in the travelling wave solutions satisfying the Rankine--Hugoniot jump conditions of \eref{eq:model}.  The related Lax entropy condition for physically relevant jumps (with non-decreasing entropy) is satisfied provided the layer flow is from $S_r$ to $S_a$ and not vice versa.  This is discussed in more detail in \cite{Marchant_Norbury_Perumpanani_00, Wechselberger_Pettet_10, Harley_vanHeijster_Marangell_Pettet_Wechselberger_13}.  
\end{remark}

\subsection{Reduced problem}

Next, we analyse the reduced problem \eref{eq:reduced} and observe that the three algebraic constraints are equivalent to the steady states of the layer problem.  In other words, the slow flow of the reduced problem is restricted to $S$.  We analyse \eref{eq:reduced} by investigating the solution behaviour in the $(u,w)$-phase plane.  

The reduced problem can be written purely in terms of the original variables $u$ and $w$ by substituting the expressions for $\hat{u}$, $v$ and $\hat{w}$ in \eref{eq:S} into the differential equations of \eref{eq:reduced} (which also projects the flow of the reduced problem onto $S$).  Thus, the two-dimensional system describing the reduced flow is
\begin{equation}\label{eq:pre-ds}
M\left[\begin{array}{c} u_z \\ w_z \end{array}\right] \coloneqq \left[\begin{array}{cc} c & 0 \\ \dfrac{\strut w}{\strut c}f(2u,w) & c + \dfrac{\strut u}{\strut c}f(u,2w) \end{array}\right]\left[\begin{array}{c} u_z \\ w_z \end{array}\right] = \left[\begin{array}{c} - u f(u,w) \\ -\alpha w(\beta u - 1) \end{array}\right]. 
\end{equation}

\begin{lemma}\label{lemma:reducedflow}
The flow of the reduced problem \eref{eq:pre-ds} is topologically equivalent to the flow of the desingularised system 
\begin{equation}\label{eq:ds}\begin{aligned}
\diff{u}{\bar{z}} &= -\frac{u f(u,w)\left(c^2 + u f(u,2w)\right)}{c}, \\
\diff{w}{\bar{z}} &= \frac{uw f(u,w)f(2u,w)}{c} - \alpha cw(\beta u - 1), \\
\end{aligned}\end{equation}
up to a scaling of the independent variable $\bar{z}$.  More specifically, the flow of \eref{eq:pre-ds} and \eref{eq:ds} is equivalent in forward $\bar{z}$ on $S_a$ and equivalent in backward $\bar{z}$ on $S_r$.  
\end{lemma}

\begin{proof}
The matrix $M$ is singular along $c^2 + u f(u,2w) = 0$, which corresponds to the fold curve $F$.  To remove this singularity we left-multiply the system by the cofactor matrix of $M$ to give
\[ \left(c^2 + u f(u,2w)\right)\left[\begin{array}{c} u_z \\ w_z \end{array}\right] = \left[\begin{array}{cc} -\dfrac{u}{c} f(u,w)\left(c^2 + u f(u,2w)\right) \\ \dfrac{uw}{c} f(u,w)f(2u,w) - \alpha cw(\beta u - 1) \end{array}\right], \]
and then rescale the independent variable $z$ via
\[ \diff{z}{\bar{z}} = c^2 + u f(u,2w). \]
This gives system \eref{eq:ds}.  Since the flow of the reduced problem is projected onto $S$, the region of the $(u,w)$-phase plane for which $w > F(u)$ corresponds to $S_r$ and the region for which $w < F(u)$ corresponds to $S_a$.  The expression for $z_{\bar{z}}$ vanishes exactly on the fold curve $F$, which corresponds to a change of stability of $S$ as the third eigenvalue of the linearisation of the layer problem $\lambda_3$ passes through zero; see Appendix~\ref{ap:layer}.  Moreover, for $w < F(u)$ (on $S_a$), $z_{\bar{z}} > 0$ and for $w > F(u)$ (on $S_r$), $z_{\bar{z}} < 0$.  Therefore, the flow of \eref{eq:ds} is equivalent to the flow of \eref{eq:pre-ds} on $S_a$ and differs only by sign, or direction, on $S_r$.  This completes the proof of the lemma.  
\qed\end{proof}

\subsection{Equilibria of the desingularised system}

As the desingularised system \eref{eq:ds} does not contain any singularities, the phase plane analysis is more straightforward.  Rather, the expression $c^2 + u f(u,2w)$ is now present on the right hand side of the $u$ ODE due to the rescaling and so the wall of singularities or fold curve \eref{eq:wall} appears as a $u$-nullcline in \eref{eq:ds}.  

The equilibrium points of \eref{eq:ds} are the original background states of \eref{eq:model-with-diff},
\[ (u_T,w_T) = (0,0), \quad (u_W,w_W) = (1,0), \quad (u_H,w_H) = \left(\frac{1}{\beta},1 - \frac{1}{\beta}\right) \]
and, in addition,
\[ (u_{C_0^{\pm}},w_{C_0^{\pm}}) = \left(\frac{1}{2}\left(1 \pm \sqrt{1 + 4c^2}\right), 0\right), \quad (u_{C_k},w_{C_k}) = \left(u_{C_k},F(u_{C_k})\right), \quad k = 1,\ldots,4, \]
where the $u_{C_k}$ are the roots of \eref{eq:poly}:
\[ 3u^4 - 4u^3 + \left[1 + 4c^2(1 - \alpha\beta)\right]u^2 + 2c^2(2\alpha - 1)u + c^4 = 0. \]
We remark that all the additional equilibria lie on $F$.  Consequently, they are folded singularities of \eref{eq:pre-ds}; see Section~\ref{subsec:recovering-reduced-problem}.  

Since we are only interested in equilibria of \eref{eq:ds} that lie in the positive quadrant, we can immediately ignore $(u_{C_0^-},w_{C_0^-})$ and for notational convenience drop the $+$ subscript in the corresponding positive equilibria.  That is, henceforth $(u_{C_0},w_{C_0})$ refers to $(u_{C_0^+},w_{C_0^+})$.  Likewise, we need to determine which of the four roots of \eref{eq:poly} are positive, and further, which of these lead to positive values for the corresponding $w_{C_k}$.  Thus, since $w = F(u)$ is a monotonically decreasing function of $u$, we are interested in roots of \eref{eq:poly} for which $0 < u_{C_k} < u_{C_0}$.  Note that since $u = 0$ and $w = 0$ are invariant sets, trajectories cannot leave the positive quadrant.  

The Jacobian of \eref{eq:ds} is 
\[ J = \left[ \begin{array}{cc} j_{11} & j_{12} \\ j_{21} & j_{22} \end{array} \right], \]
where
\[ \begin{aligned}
j_{11} &= -\frac{(c^2 + u f(u,2w))f(2u,w)}{c} - \frac{u f(u,w)f(2u,2w)}{c}, \\
j_{12} &= \frac{u(c^2 + u f(u,2w))}{c} + \frac{2u^2 f(u,w)}{c}, \\
j_{21} &= \frac{w f(2u,w)^2}{c}  - \frac{2uw f(u,w)}{c} - \alpha\beta cw, \\
j_{22} &= \frac{u f(2u,w)f(u,2w)}{c} - \frac{uw f(u,w)}{c} - \alpha c(\beta u - 1). 
\end{aligned} \]

Linear analysis reveals that since $\beta > 1$, both $(u_T,w_T)$ and $(u_W,w_W)$ are saddles.  Furthermore, since $u_{C_0} > 1$, $(u_W,w_W)$ will always lie on $S_a$.  For the healed state we find that
\[ (u_H,w_H) \quad {\rm is}\quad \begin{cases} \textrm{a saddle} & {\rm for}\quad 0 < c < c_1(\beta), \\ \textrm{an unstable node} & {\rm for}\quad c_1(\beta) < c < c_2(\alpha,\beta), \\ \textrm{an unstable focus} & {\rm for}\quad c > c_2(\alpha,\beta), \end{cases} \]
where 
\[ c_1(\beta) = \frac{\sqrt{\beta - 1}}{\beta} > 0 \quad {\rm and} \quad c_2(\alpha,\beta) = \frac{2\sqrt{\alpha}(\beta - 1)}{\sqrt{\beta(4\alpha\beta(\beta - 1) - 1)}}.  \]
Note that $c_2(\alpha,\beta)$ is complex for $\alpha < (4\beta(\beta - 1))^{-1}$.  So for certain parameter choices the transition from an unstable node to an unstable focus does not occur.  The transition from a saddle to an unstable node (at $c = c_1(\beta)$) occurs as $(u_H,w_H)$ crosses over $F$, from $S_r$ to $S_a$, and is independent of $\alpha$.  Thus, for $c = c_1(\beta)$ one of the roots of \eref{eq:poly} coincides with $(u_H,w_H)$.  (The curves $c = c_1(\beta)$ and $c = c_2(\alpha,\beta)$ for various values of $\beta$ are shown in Figure~\ref{fig:type_pos_canards}.)

For $(\alpha,c)$ to the left of the curve $c = c_3(\alpha;\beta)$, $(u_{C_0},w_{C_0})$ is a saddle and for $(\alpha,c)$ to the right it is a stable node, with
\[ c_3(\alpha;\beta) = \begin{cases} \dfrac{\sqrt{(1 - \alpha)(\alpha(\beta - 1) - 1)}}{\alpha\beta - 2} & \textrm{for}\quad \dfrac{2}{\beta} < \alpha \leq \dfrac{1}{\beta - 1}, \quad \beta \leq 2, \\[0.2cm] \dfrac{\sqrt{(1 - \alpha)(\alpha(\beta - 1) - 1)}}{2 - \alpha\beta} & \textrm{for}\quad \dfrac{1}{\beta - 1} \leq \alpha < \dfrac{2}{\beta}, \quad \beta > 2; \end{cases} \]
see also Figure~\ref{fig:type_pos_canards}.  By comparing the gradient of the wall of singularities as it crosses the $u$-axis with the non-trivial eigenvector of the linearised system at $(u_{C_0},w_{C_0})$, we can conclude that when $(u_{C_0},w_{C_0})$ is a saddle, its unstable manifold $\mathcal{W}^U(u_{C_0},w_{C_0})$ enters the positive quadrant on $S_r$, whereas when it is a stable node, the now stable manifold $\mathcal{W}^S(u_{C_0},w_{C_0})$ enters the positive quadrant on $S_a$ (in backward $\bar{z}$).  Furthermore, at $c = c_3(\alpha;\beta)$ one of the roots of \eref{eq:poly} coincides with $(u_{C_0},w_{C_0})$.  

The remaining equilibria are determined by the roots of \eref{eq:poly}, which, in principle, can be solved exactly.  However, it is impossible to determine which roots are real and positive from their analytic expressions for generic parameter values.  Nonetheless, we can say something about the maximum number of positive roots of \eref{eq:poly} using Descartes' rule of sign; see, for example, \cite{Anderson_Jackson_Sitharam_98}.  

\begin{lemma}\label{lemma:rule-of-sign}
If $\alpha > \dfrac{1}{2}$ or $\beta > \dfrac{1}{\alpha}\left(1 + \dfrac{1}{4c^2}\right)$, \eref{eq:poly} has a maximum of two positive roots.  
\end{lemma}

\begin{proof}
Descartes' rule of sign states that the maximum number of positive roots of a polynomial is determined by the number of sign changes between consecutive coefficients.  Consequently, for the fourth order polynomial in question, the only regime where we have a maximum of four positive roots is when the $u^2$-coefficient is positive and the $u$-coefficient is negative.  Thus, to have a maximum of four positive roots we require 
\[ 1 + 4c^2(1 - \alpha\beta) > 0 \quad {\rm and} \quad 2c^2(2\alpha - 1) < 0. \]
In all other cases we have a maximum of two positive roots, which yields the required result.  
\qed\end{proof}

Note that in theory we could obtain further information about the number of positive roots of \eref{eq:poly}, or more specifically, the number of roots in the interval $(0,u_{C_0})$ using Sturm's theorem; see, for example, \cite{Akritas_Vigklas_10}.  However, in practice this theorem provides no more useful information for general parameter values than the exact solution.  Therefore, we instead solve \eref{eq:poly} over a range of parameter values using {\tt MATLAB}'s numerical root finding algorithm {\tt roots} and count the number of roots $u_{C_k} \in (0,u_{C_0})$.  For each set of parameter values we also compute the eigenvalues of the associated Jacobian to determine the (linear) stability of each equilibrium of \eref{eq:ds}.  The results are presented in Figure~\ref{fig:type_pos_canards}.  We remark that as $\beta$ is increased further than shown in Figure~\ref{fig:type_pos_canards}, the results remain qualitatively the same.  These results illustrate that within the chosen ranges of the parameters, we can expect to see up to two equilibria $(u_{C_k},w_{C_k})$ in the positive quadrant and we never observe four positive roots with $0 < u_{C_k} < u_{C_0}$.  

\begin{figure}[ht]
\centering
\subfloat[]{\label{fig:type-A}\includegraphics[width=0.5\textwidth]{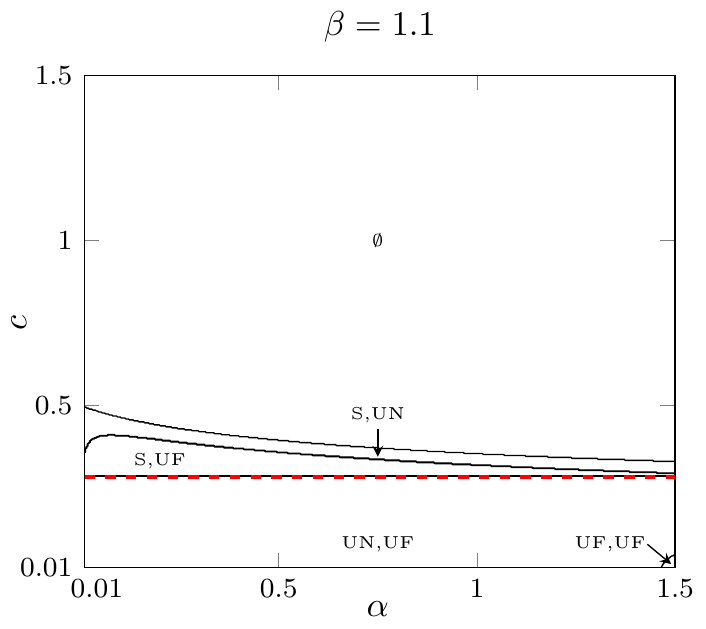}}
\subfloat[]{\label{fig:type-B}\includegraphics[width=0.5\textwidth]{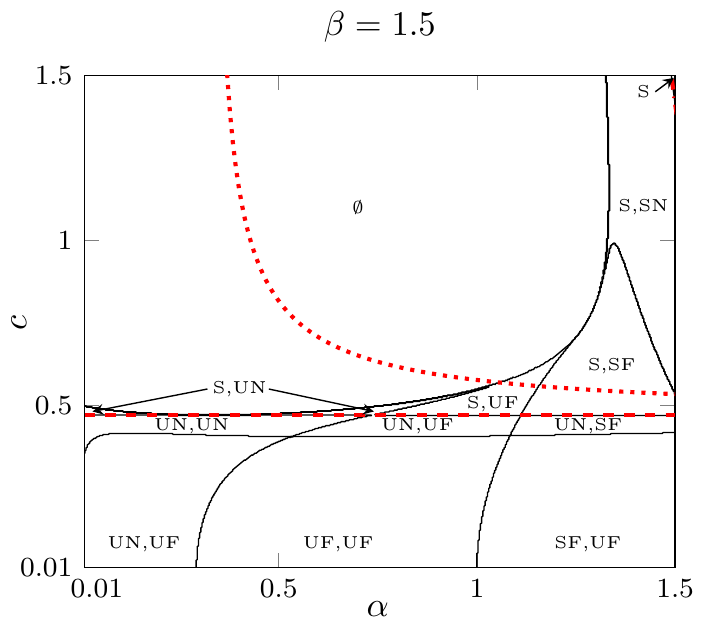}}\\
\subfloat[]{\label{fig:type-C}\includegraphics[width=0.5\textwidth]{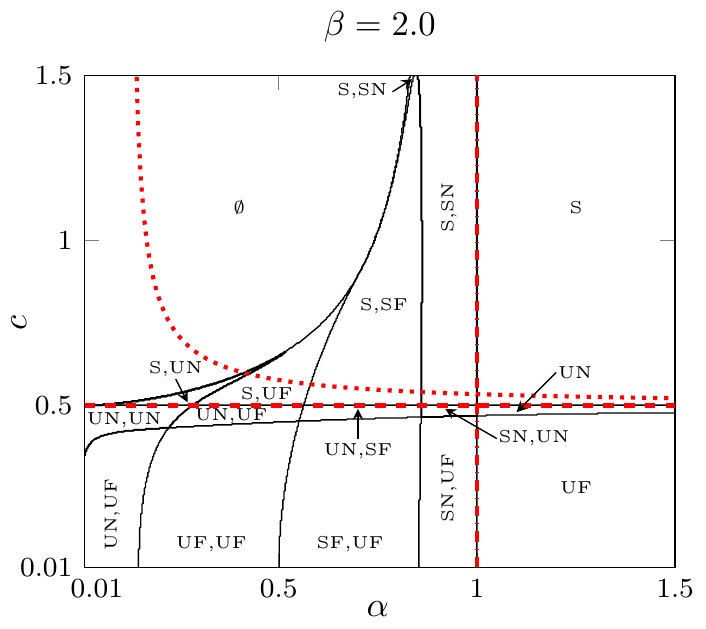}}
\subfloat[]{\label{fig:type-D}\includegraphics[width=0.5\textwidth]{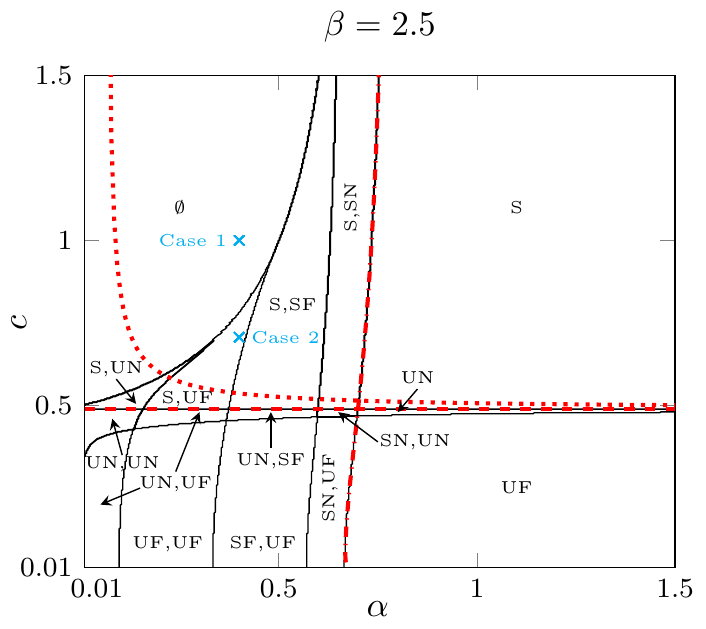}}\\
\includegraphics{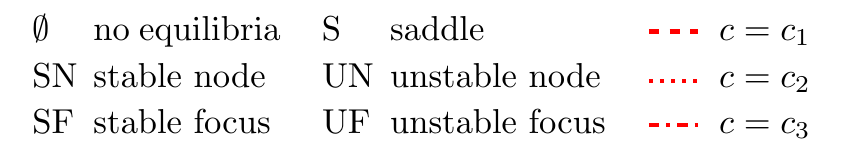}
\caption{Illustration of the number and stability of the equilibria $(u_{C_k},w_{C_k})$ of \eref{eq:ds}, with $0 < u_{C_k} < u_{C_0}$, for different parameter values.  The points corresponding to the locations of the Case 1 and Case 2 parameter regimes are marked $\times$.  In both cases $(u_H,w_H)$ is an unstable focus since $c > c_2(\alpha,\beta)$.  Note that for the parameter ranges shown here, we do not observe any saddle equilibria for $c < c_1(\beta)$.}
\label{fig:type_pos_canards}
\end{figure}

\subsection{Constructing phase planes for the desingularised system}

Due to the variability in the number and stability of the equilibria of \eref{eq:ds}, in particular the ones arising from the solutions of \eref{eq:poly}, it is infeasible to construct phase planes of \eref{eq:ds} for general parameter values.  Therefore, we select the two specific parameter regimes discussed in Section~\ref{subsec:previous-results} (see \eref{eq:cases}), with the locations in parameter space illustrated in Figure~\ref{fig:type-D}.  Under both parameter regimes $(u_H,w_H) = (2/5,3/5)$ lies on $S_a$ and is an unstable focus since $c > c_2(2/5,5/2) = 6\sqrt{5}/25 \approx 0.54$.  

\subsubsection{Case 1}
\label{subsec:case1-ds}

Firstly, we consider the case where $\alpha = 2/5$, $\beta = 5/2$ and $c = 1$.  Figure~\ref{fig:type_pos_canards} suggests that in this case there are no equilibria $(u_{C_k},w_{C_k})$ of \eref{eq:ds} in the positive quadrant, rather, all four roots are complex; see Table~\ref{tab:steady-states}.  The phase plane of \eref{eq:ds} for Case 1 is given in Figure~\ref{fig:case1-dspp}.  This figure suggests that \eref{eq:ds} possesses a heteroclinic orbit connecting $(u_H,w_H)$ and $(u_W,w_W)$, under this parameter regime.  

\begin{table}[t]
\begin{center}
\small
\begin{tabular}{cc|cc|cc}
\toprule
\multicolumn{2}{c}{{\bf Label}} & \multicolumn{2}{|c}{{\bf Case 1} ($\alpha = 2/5$, $\beta = 5/2$, $c = 1$)} & \multicolumn{2}{|c}{{\bf Case 2} ($\alpha = 2/5$, $\beta = 5/2$, $c = \sqrt{2}/2$)} \\
\midrule
$T$ & $(u_T,w_T)$ & $(0,0)$ & saddle & $(0,0)$ & saddle \\[6pt]
$W$ & $(u_W,w_W)$ & $(1,0)$ & saddle & $(1,0)$ & saddle \\[6pt]
$H$ & $(u_H,w_H)$ & $(0.4,0.6)$ & focus (U) & $(0.4,0.6)$ & focus (U) \\[6pt]
$C_0$ & $(u_{C_0},w_{C_0})$ & $(1.62,0)$ & saddle & $(1.37,0)$ & saddle \\[6pt]
$C_0^-$ & $(u_{C_0^-},w_{C_0^-})$ & \textcolor{gray}{$(-0.62,0)$} & \textcolor{gray}{saddle} & \textcolor{gray}{$(-0.37,0)$} & \textcolor{gray}{saddle} \\[6pt]
$C_1$ & $(u_{C_1},w_{C_1})$ & \textcolor{gray}{$(0.93+0.32{\rm i},0.52-0.32{\rm i})$} & \textcolor{gray}{-} & $(0.97,0.27)$ & focus (U) \\[6pt]
$C_2$ & $(u_{C_2},w_{C_2})$ & \textcolor{gray}{$(0.93-0.32{\rm i},0.52+0.32{\rm i})$} & \textcolor{gray}{-} & $(0.62,0.59)$ & saddle \\[6pt]
$C_3$ & $(u_{C_3},w_{C_3})$ & \textcolor{gray}{$(-0.26+0.53{\rm i},0.25-1.02{\rm i})$} & \textcolor{gray}{-} & \textcolor{gray}{$(-0.13+0.35{\rm i}, 0.33-0.81{\rm i})$} & \textcolor{gray}{-} \\[6pt]
$C_4$ & $(u_{C_4},w_{C_4})$ & \textcolor{gray}{$(-0.26-0.53{\rm i},0.25+1.02{\rm i})$} & \textcolor{gray}{-} & \textcolor{gray}{$(-0.13-0.35{\rm i}, 0.33+0.81{\rm i})$} & \textcolor{gray}{-} \\
\bottomrule
\end{tabular}
\end{center}
\caption{Locations and type of the equilibrium points of \eref{eq:ds} under the two parameter regimes.  The greyed out equilibria lie outside the first quadrant.}
\label{tab:steady-states}
\end{table}

\begin{figure}[ht]
\centering
\subfloat[Case 1]{\label{fig:case1-dspp}\includegraphics[width=0.5\textwidth]{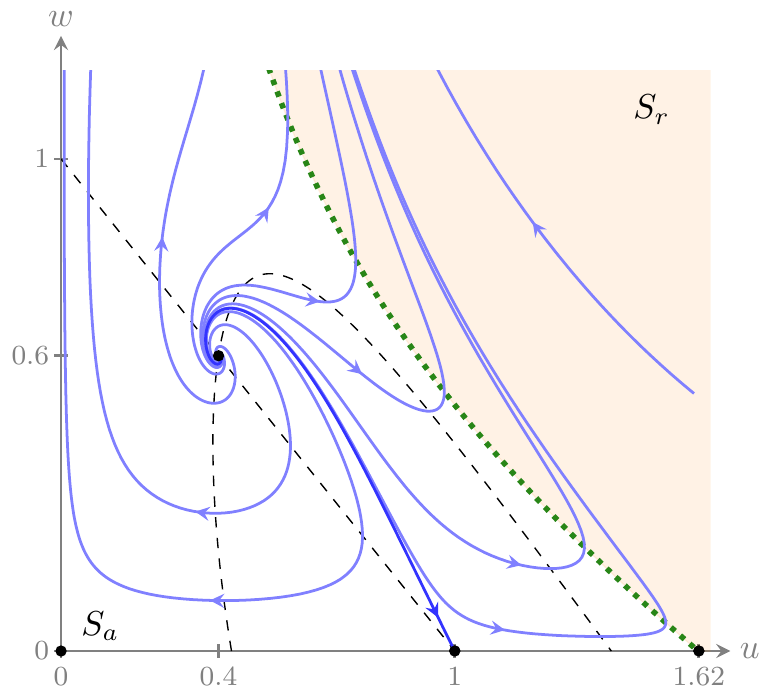}}
\subfloat[Case 2]{\label{fig:case2-dspp}\includegraphics[width=0.5\textwidth]{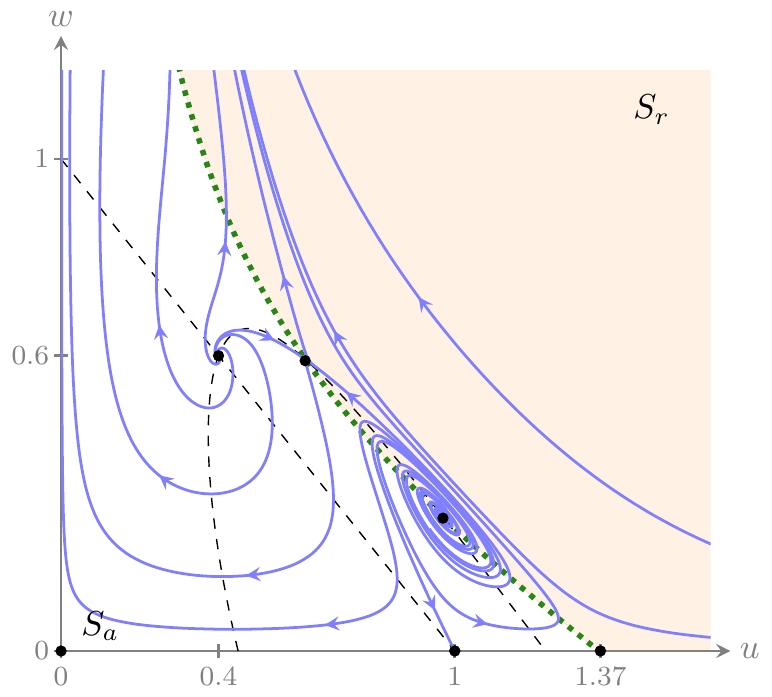}}
\caption{Phase planes of \eref{eq:ds} under the two parameter regimes.  The green dotted line and the black dashed lines correspond to the nullclines of the system and the black circles to the equilibria.  Note that both axes are also nullclines.  Moreover, the green dotted line coincides with $F$.}
\label{fig:ds-phase-planes}
\end{figure}

Although we do not observe any limit cycles in the numerically generated phase planes shown in Figure~\ref{fig:ds-phase-planes}, {\em a priori} we cannot exclude the appearance of limit cycles.  Therefore, we conjecture:

\begin{conjecture}\label{conj:limitcycles}
System \eref{eq:ds} possesses no limit cycles under the parameter regimes Case 1 or Case 2.  
\end{conjecture}

\begin{lemma}\label{lemma:case1}
Assume Conjecture~\ref{conj:limitcycles} holds and that the system parameters are as in Case 1.  Then, \eref{eq:ds} possesses a heteroclinic orbit connecting $(u_H,w_H)$ with $(u_W,w_W)$.  Moreover, this orbit does not cross $F$.  
\end{lemma}

\begin{proof}
Recall that $u = 0$ and $w = 0$ are invariant sets and that $w = F(u)$ is a $u$-nullcline, along which $w_{\bar{z}} > 0$ for $w > 0$.  Furthermore, for $w = R$ with $R$ sufficiently large and $u \geq 0$, $w_{\bar{z}} > 0$.  Therefore, we have that trajectories must leave the region $\mathcal{R}$ bounded by the curves $u = 0$, $w = 0$, $w = R$ and $w = F(u)$; see Figure~\ref{fig:case1-monotonicity} where the unshaded region corresponds (up to $w = R$) to $\mathcal{R}$.  The Poincar\'e--Bendixson theorem \cite{Jordan_Smith_07} then implies that the trajectory leaving $(u_W,w_W)$ in backward $\bar{z}$ must approach either $(u_T,w_T)$, $(u_H,w_H)$ or $(u_{C_0},w_{C_0})$, or a limit cycle.  The latter is excluded by Conjecture~\ref{conj:limitcycles}.  We also exclude connections to $(u_T,w_T)$ or $(u_{C_0},w_{C_0})$ since both are saddles and their stable manifolds in backward $\bar{z}$ are not inside $\mathcal{R}$.  Thus, $(u_W,w_W)$ connects to $(u_H,w_H)$ in backward $\bar{z}$.  
\qed\end{proof}

\begin{corollary}\label{cor:case1}
Assume Conjecture~\ref{conj:limitcycles} holds, \eref{eq:poly} has no real solutions and $(u_H,w_H)$ lies on $S_a$.  Then \eref{eq:ds} possesses a heteroclinic orbit connecting $(u_H,w_H)$ with $(u_W,w_W)$.  
\end{corollary}

\begin{proof}
This follows immediately from the proof of Lemma~\ref{lemma:case1} since the region $\mathcal{R}$ is constructed without imposing any further conditions on the parameters.  
\qed\end{proof}

\begin{remark}\label{remark:FSN-II}
Under our original parameter restraints, $\alpha > 0$, $\beta > 1$ and $c > 0$, $(u_H,w_H)$ living on $S_r$ implies that \eref{eq:poly} has at least one real root.  Thus, the restriction that $(u_H,w_H)$ lies on $S_a$ in Corollary~\ref{cor:case1} is not strictly necessary as it is guaranteed by the condition that \eref{eq:poly} has no real roots.  
\end{remark}

\begin{figure}[ht]
\centering
\subfloat[Case 1]{\label{fig:case1-monotonicity}\includegraphics[width=0.5\textwidth]{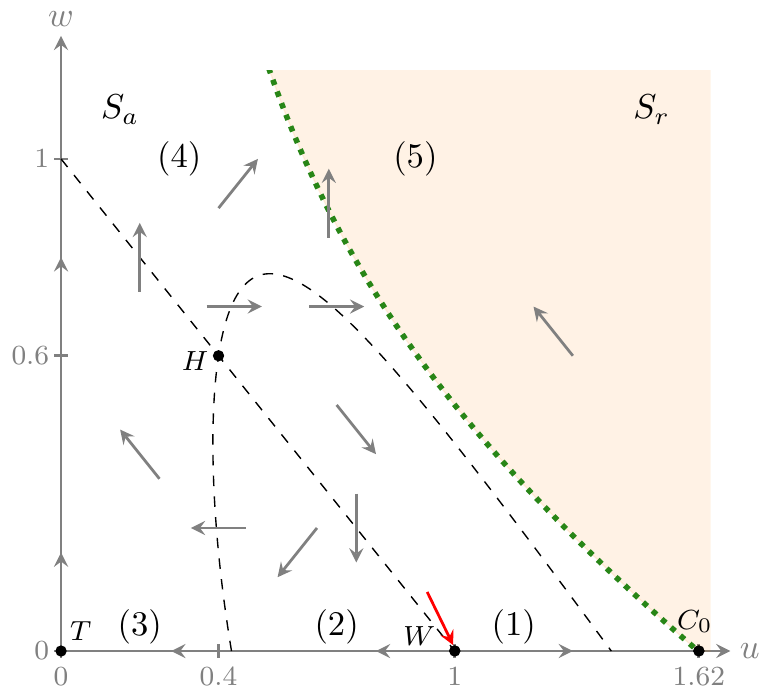}}
\subfloat[Case 2]{\label{fig:case2-monotonicity}\includegraphics[width=0.5\textwidth]{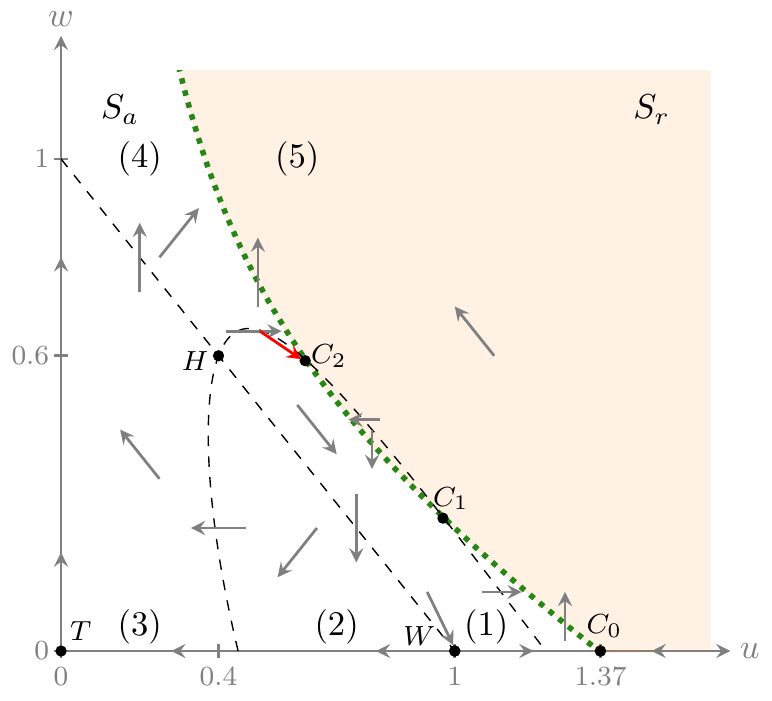}}
\caption{Solution direction of the desingularised system \eref{eq:ds} within the different regions of the phase planes, separated by the nullclines.  Once again, the green dotted line and the black dashed lines correspond to the nullclines of \eref{eq:ds} and the black circles to the equilibria, with the green dotted line also coinciding with the wall of singularities \eref{eq:wall}.  The red arrows indicate the directions of the stable eigenvector of the linearised system at the equilibria of interest.  The equilibria are labelled according to Table~\ref{tab:steady-states}.}
\label{fig:monotonicity}
\end{figure}

\subsubsection{Case 2}
\label{subsec:case2-ds}

Secondly, we investigate the case where $\alpha = 2/5$, $\beta = 5/2$ and $c = \sqrt{2}/2$.  Figure~\ref{fig:type_pos_canards} suggests that in this case \eref{eq:ds} has two additional $(u_{C_k},w_{C_k})$ equilibria in the positive quadrant: a saddle and an unstable focus.  The solutions of \eref{eq:poly} confirm that by decreasing $c$ from $1$ to $\sqrt{2}/2$, two complex roots merge and form two real, positive roots; see Table~\ref{tab:steady-states}.  The phase plane of \eref{eq:ds} for Case 2 is provided in Figure~\ref{fig:case2-dspp}.  This figure demonstrates that the smooth connection between $(u_H,w_H)$ and $(u_W,w_W)$ that was visible in Figure~\ref{fig:case1-dspp} is no longer present.  Instead, Figure~\ref{fig:case2-dspp} shows heteroclinic orbits connecting $(u_H,w_H)$ with $(u_{C_2},w_{C_2})$, $(u_{C_2},w_{C_2})$ with $(u_{C_1},w_{C_1})$ and $(u_{C_1},w_{C_1})$ with $(u_W,w_W)$.  Moreover, Figure~\ref{fig:case2-dspp} suggests that \eref{eq:ds} does not possess any limit cycles in Case 2, in accordance with Conjecture~\ref{conj:limitcycles}.  

\begin{lemma}\label{lemma:case2}
Assume Conjecture~\ref{conj:limitcycles} holds and that the system parameters are as in Case 2.  Then, a heteroclinic orbit connecting $(u_H,w_H)$ to $(u_{C_2},w_{C_2})$ exists.  
\end{lemma}

\begin{proof}
It can be shown that with the parameters as in Case 2, the trajectory entering $(u_{C_2},w_{C_2})$ along the stable eigenvector does so from region 1 of Figure~\ref{fig:case2-monotonicity}, with the stable eigenvector indicated by the red arrow.  As this trajectory is traced backwards, it passes through regions 4, 3, 2 and then back to 1 due to the directions of the solution trajectories in the various regions of the phase plane, illustrated in Figure~\ref{fig:case2-monotonicity}.  This process repeats forever, since $(u_H,w_H)$ is a focus.  Therefore, as we assumed there are no limit cycles, the solution trajectory leaving $(u_{C_2},w_{C_2})$ in backward $\bar{z}$ will spiral around $(u_H,w_H)$ in an anti-clockwise direction, approaching this end state.  This completes the heteroclinic orbit.  
\qed\end{proof}

We cannot rigorously prove that the latter two heteroclinic connections ($(u_{C_2},w_{C_2})$ to $(u_{C_1},w_{C_1})$ and $(u_{C_1},w_{C_1})$ to $(u_W,w_W)$) exist.  Therefore, we make the conjecture (which is needed in Section~\ref{subsec:case2}):

\begin{conjecture}\label{conj:hits-wall}
Under parameter regime Case 2, \eref{eq:ds} possesses heteroclinic orbits connecting $(u_{C_2},w_{C_2})$ to $(u_{C_1},w_{C_1})$ and $(u_{C_1},w_{C_1})$ to $(u_W,w_W)$.  
\end{conjecture}

Note that the numerically generated trajectories shown in Figure~\ref{fig:case2-dspp} strongly suggest that Conjecture~\ref{conj:hits-wall} is valid.  

\subsection{Recovering the flow of the reduced problem}
\label{subsec:recovering-reduced-problem}

Recall from Lemma~\ref{lemma:reducedflow} that the flow of \eref{eq:pre-ds} and \eref{eq:ds} is equivalent in forward $\bar{z}$ on $S_a$ and equivalent in backward $\bar{z}$ on $S_r$.  Consequently, the $(u,w)$-phase plane parameterised by $z$ is equivalent to the one parameterised by $\bar{z}$, except that the direction of the trajectories are reversed on $S_r$; see Figure~\ref{fig:z-pps} (in comparison with Figure~\ref{fig:ds-phase-planes}).  

\begin{figure}[ht]
\centering
\subfloat[Case 1]{\label{fig:case1-pp}\includegraphics[width=0.5\textwidth]{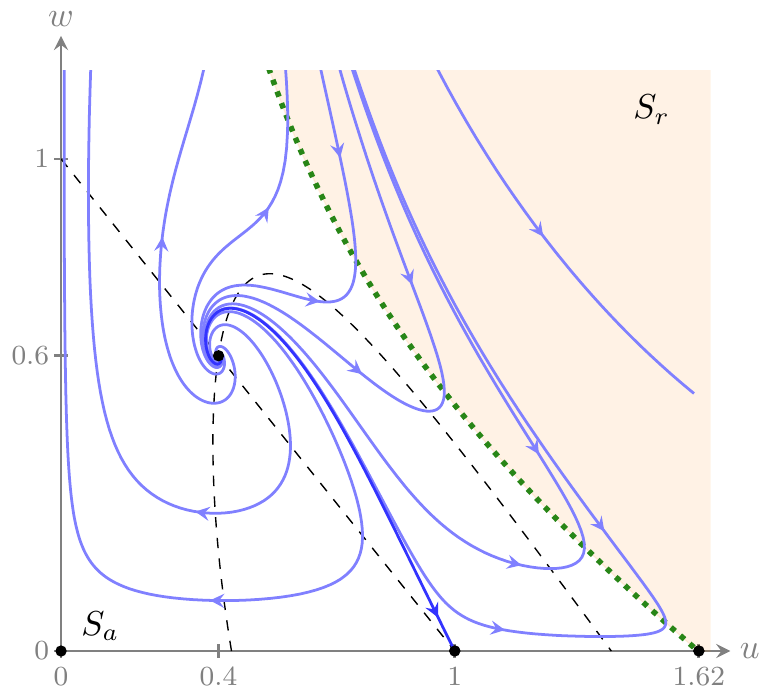}}
\subfloat[Case 2]{\label{fig:case2-pp}\includegraphics[width=0.5\textwidth]{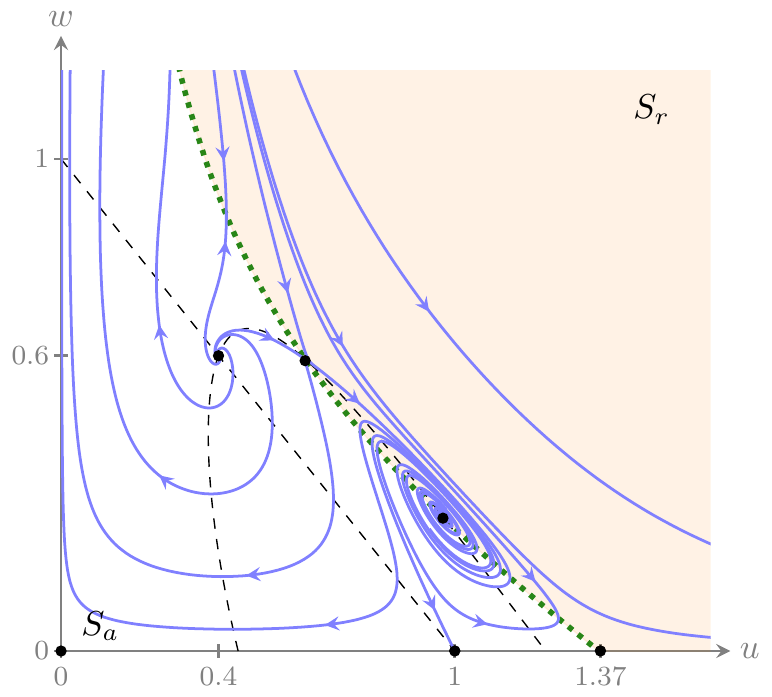}}
\caption{The $(u,w)$-phase plane of the reduced problem \eref{eq:pre-ds}, parameterised by $z$, under the two parameter regimes.  The green dotted line corresponds to the wall of singularities or fold curve $F$.}
\label{fig:z-pps}
\end{figure}

Importantly, as the direction of the trajectories on $S_r$ are reversed, the equilibria of \eref{eq:ds} that lie on $F$ become {\em folded singularities} or {\em canard points} of \eref{eq:pre-ds}.  These points are not equilibria of \eref{eq:pre-ds} but are equivalent to the gates or holes in the wall of singularities, as discussed in Section~\ref{subsec:main-result}; see also \cite{Wechselberger_Pettet_10}.  The existence of canard points is fundamental to the existence of heteroclinic orbits, and hence travelling wave solutions, in regions of parameter space where the wall of singularities prevents a smooth connection between the end states of the wave, such as Case 2.  

We emphasise key results from the analysis of \eref{eq:ds}, reinterpreted for \eref{eq:pre-ds}, for general parameters.  
\begin{itemize}
\item The equilibria $(u_T,w_T)$ and $(u_W,w_W)$ are always located on $S_a$.  In the limit $c \to 0$, $(u_{C_0},w_{C_0})$ approaches $(u_W,w_W)$ but they do not cross.  
\item The sets $u = 0$ and $w = 0$ are invariant.  
\item At $c = c_1(\beta)$, we observe a folded saddle-node type II (FSN II) bifurcation \cite{Krupa_Wechselberger_10}.  As $c$ decreases through $c_1(\beta)$, the equilibrium $(u_H,w_H)$ crosses over $F$ via one of the $(u_{C_k},w_{C_k})$ canard points causing $(u_H,w_H)$ to transition from an unstable node to a saddle and the relevant $(u_{C_k},w_{C_k})$ to transition from a folded saddle to a folded node; see Figure~\ref{fig:type_pos_canards}.  Thus, the FSN II bifurcation can also be regarded as a transcritical bifurcation between the equilibrium $(u_H,w_H)$ and the relevant canard point $(u_{C_k},w_{C_k})$.  Note that this implies that for $c < c_1(\beta)$ (and hence $(u_H,w_H)$ on $S_r$), \eref{eq:poly} will have a least one real root, in accordance with Remark~\ref{remark:FSN-II}.  
\item At $c = c_3(\alpha;\beta)$, we observe a folded saddle-node type I (FSN I) bifurcation \cite{Krupa_Wechselberger_10}.  As $w_{C_k}$ decreases through $0$ (or $c$ passes through $c_3(\alpha;\beta)$), one of the $(u_{C_k},w_{C_k})$ canard points crosses over $w = 0$ via $(u_{C_0},w_{C_0})$ causing $(u_{C_0},w_{C_0})$ to transition from a folded saddle to a folded node and the relevant $(u_{C_k},w_{C_k})$ to transition from a folded node to a folded saddle; see Figure~\ref{fig:type_pos_canards}.  Thus, the FSN I bifurcation can also be regarded as a transcritical bifurcation between the canard points $(u_{C_0},w_{C_0})$ and the relevant $(u_{C_k},w_{C_k})$.  
\item At $c = \tilde{c}(\alpha,\beta)$, we observe a FSN I bifurcation, that is, two canard points bifurcating in a saddle-node bifurcation.  In Figure~\ref{fig:type_pos_canards}, the curve $c = \tilde{c}(\alpha,\beta)$ is visible as the lower boundary of the region labelled $\emptyset$.  
\end{itemize}

Using the results from this section as well as those from Section~\ref{subsec:layer}, we now construct {\em singular} heteroclinic orbits ($\varepsilon = 0$) for the two parameter regimes of interest.  Then, using Fenichel theory and canard theory, we prove their persistence for $0 < \varepsilon \ll 1$.  

\subsection{A travelling wave solution for Case 1}
\label{subsec:case1}

\begin{theorem}\label{theorem:case1}
Assume Conjecture~\ref{conj:limitcycles} holds.  If the system parameters are as in Case 1, then there exists an $\varepsilon_0 > 0$ such that for $\varepsilon \in [0,\varepsilon_0]$, \eref{eq:model-with-diff} possesses a smooth travelling wave solution, connecting $(u_H,w_H)$ to $(u_W,w_W)$.  
\end{theorem}

\begin{proof}
Since the equilibria $(u_T,w_T)$, $(u_W,w_W)$ and $(u_H,w_H)$ all lie on $S_a$, their stability is not affected by reversing the direction of the trajectories on $S_r$.  Moreover, the heteroclinic orbit in Lemma~\ref{lemma:case1} connects $(u_H,w_H)$ to $(u_W,w_W)$ while remaining on $S_a$.  Therefore, \eref{eq:pre-ds} also possesses a heteroclinic orbit connecting $(u_H,w_H)$ with $(u_W,w_W)$.  This orbit and the corresponding wave shape are given in Figure~\ref{fig:case1-sols}.  

\begin{figure}[ht]
\centering
\subfloat[Case 1: Smooth wave.]{\label{fig:case1-sols}\includegraphics{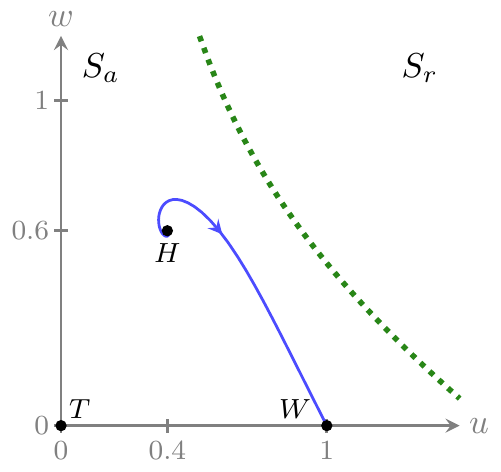}\includegraphics{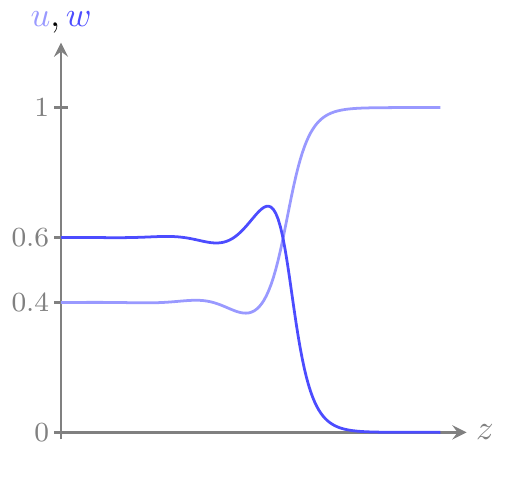}} \\
\subfloat[Case 2: Shock-fronted wave.]{\label{fig:case2-sols}\includegraphics{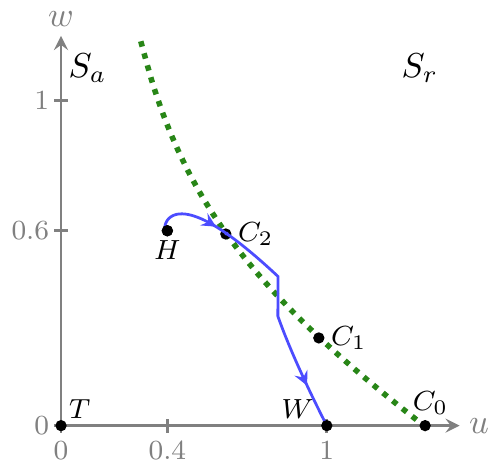}\includegraphics{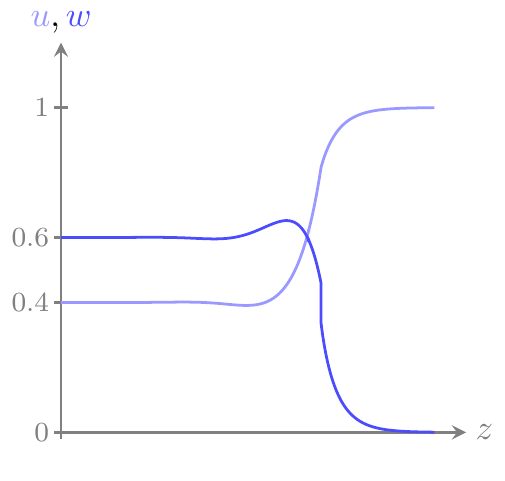}}
\caption{Heteroclinic orbits in the $(u(z),w(z))$-phase plane and the corresponding wave shapes for the two parameter regimes Case 1 and Case 2.  The labels in the left hand panels refer to Table~\ref{tab:steady-states}.}
\label{fig:sols}
\end{figure}

We use Fenichel theory \cite{Fenichel_71, Fenichel_79} to prove that this {\em singular} orbit perturbs to a nearby orbit of the full system \eref{eq:model-with-diff} for sufficiently small $0 < \varepsilon \ll 1$.  The arguments here are equivalent to those presented in \cite{Harley_vanHeijster_Marangell_Pettet_Wechselberger_13} for the so-called {\em Type I} waves and so we refer the reader to this work for the details.  In summary, since $S_a$ is a normally hyberbolic manifold away from the fold curve, it deforms smoothly to a locally invariant manifold $S_{a,\varepsilon}$ and the {\em singular} heteroclinic orbit identified in Lemma~\ref{lemma:case1} perturbs smoothly to an $\mathcal{O}(\varepsilon)$ close orbit on $S_{a,\varepsilon}$.  Therefore, the corresponding travelling wave solution persists as a nearby solution of \eref{eq:model-with-diff} for $0 \leq \varepsilon \ll 1$, with the parameters as in Case 1.  
\qed\end{proof}

\subsection{A travelling wave solution for Case 2}
\label{subsec:case2}

\begin{theorem}\label{theorem:case2}
Assume Conjecture~\ref{conj:limitcycles} and Conjecture~\ref{conj:hits-wall} hold.  If the system parameters are as in Case 2, then there exists an $\varepsilon_0 > 0$ such that for $\varepsilon \in [0,\varepsilon_0]$, \eref{eq:model-with-diff} possesses a travelling wave solution containing a shock (in the singular limit $\varepsilon \to 0$), connecting $(u_H,w_H)$ to $(u_W,w_W)$.  
\end{theorem}

\begin{proof}
The equilibria $(u_T,w_T)$, $(u_W,w_W)$ and $(u_H,w_H)$ still lie on $S_a$ and hence their local stable and unstable manifolds remain unaffected as the direction of the solution trajectories on $S_r$ are reversed.  However, we now have to consider what happens to $(u_{C_1},w_{C_1})$, $(u_{C_2},w_{C_2})$ and $(u_{C_0},w_{C_0})$.  As previously discussed, these points are not equilibria of the reduced problem \eref{eq:pre-ds} but rather canard points.  Thus, \eref{eq:pre-ds} possesses three equilibrium points in the positive quadrant, as well as three canard points: $(u_{C_0},w_{C_0})$, $(u_{C_1},w_{C_1})$ and $(u_{C_2},w_{C_2})$.  

Since $(u_{C_1},w_{C_1})$ is now a {\em folded focus}, no trajectories can pass through it due to the reversal of the solution directions on $S_r$ \cite{Wechselberger_12}; see also the left two panels of Figure~\ref{fig:spiral-canard}.  The other canard points $(u_{C_0},w_{C_0})$ and $(u_{C_2},w_{C_2})$ are now {\em folded saddles}.  As the direction of the trajectories on $S_r$ are reversed, the stable (unstable) eigenvector of the saddle of \eref{eq:ds} becomes unstable (stable) and so the folded saddle of \eref{eq:pre-ds} admits two trajectories: one that passes from $S_a$ to $S_r$ and the other from $S_r$ to $S_a$.  The former is referred to as the canard solution and the latter the faux-canard solution; see, for example, \cite{Wechselberger_12} as well as the right two panels of Figure~\ref{fig:spiral-canard}.  

\begin{figure}[ht]
\centering
\includegraphics[width=0.45\textwidth]{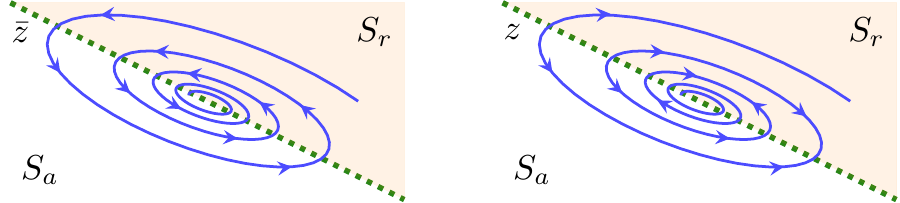}\qquad
\includegraphics[width=0.45\textwidth]{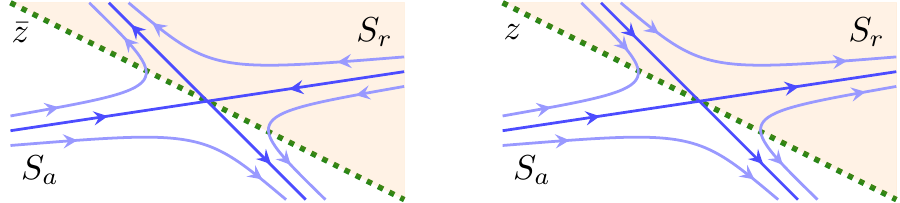}
\caption{A schematic of a folded focus canard point (no trajectories can pass through) and a folded saddle canard point (two trajectories pass through).  The trajectory that passes through the folded saddle from $S_a$ to $S_r$ is called the canard solution and the one that passes from $S_r$ to $S_a$, the faux-canard solution.}
\label{fig:spiral-canard}
\end{figure}

Lemma~\ref{lemma:case2} implies that a connection exists between $(u_H,w_H)$ and $(u_{C_2},w_{C_2})$ by Conjecture~\ref{conj:limitcycles}.  This connection is not affected by reversing the direction of the trajectories on $S_r$.  However, rather than the trajectory terminating at $(u_{C_2},w_{C_2})$, it continues through the folded saddle along the canard solution onto $S_r$.  From here, the only way the trajectory can return to $S_a$ and hence $(u_W,w_W)$, is to leave $S_r$ via the layer flow.  

Recall that the layer flow connects points on $S_r$ to points on $S_a$ with constant $u$, $\hat{u}$ and $\hat{w}$.  In the slow scaling, this flow appears as an instantaneous {\em jump} from $S_r$ to $S_a$.  Holding $u$, $\hat{u}$ and $\hat{w}$ constant along the layer flow also provides information about the value of $w$ and either end of the layer flow, or equivalently, at either end of a {\em jump}.  In particular, we have that $w = F(u) \pm C$ at the start and end of a jump, respectively, with $F(u)$ defined in \eref{eq:wall} and $C$ an arbitrary constant.  This follows from solving $\hat{w}(u,w) = \hat{w}(u,v(u,w),w)$ as defined in \eref{eq:S} for $w$, assuming $u$ and $\hat{w}$ are fixed.  If $S$ is projected onto $(u,w)$-space, as in Figure~\ref{fig:case2-pp} for example, this property corresponds to the end points of a jump being equidistant from the fold curve in $w$, while constant in $u$.  

In summary, for a heteroclinic connection between $(u_H,w_H)$ and $(u_W,w_W)$ to be possible, we must be able find a $u^{\ast}$ such that the unstable manifold of the folded saddle $\mathcal{W}^U(u_{C_2},w_{C_2})$ on $S_r$ and the stable manifold of the end state $\mathcal{W}^S(u_W,w_W)$, are equidistant from $F(u^{\ast})$.  

Conjecture~\ref{conj:hits-wall} implies that $\mathcal{W}^S(u_W,w_W)$ intersects the wall of singularities or fold curve between $(u_{C_2},w_{C_2})$ and $(u_{C_1},w_{C_1})$.  Furthermore, it implies that $\mathcal{W}^U(u_{C_2},w_{C_2})$ on $S_r$ intersects the fold curve between $(u_{C_1},w_{C_1})$ and $(u_{C_0},w_{C_0})$.  These results guarantee that the conditions for a jump are satisfied, connecting $\mathcal{W}^U(u_{C_2},w_{C_2})$ on $S_r$ to $\mathcal{W}^S(u_W,w_W)$ on $S_a$.  That is, a vertical line can be drawn in the $(u,w)$-phase plane of \eref{eq:pre-ds} shown in Figure~\ref{fig:case2-pp}, connecting $\mathcal{W}^U(u_{C_2},w_{C_2})$ to $\mathcal{W}^S(u_W,w_W)$, with the end points equidistant from the fold curve in $w$.  

Once the trajectory lands back on $S_a$ on $\mathcal{W}^S(u_W,w_W)$, it will connect to the end state, completing the heteroclinic orbit.  This trajectory and the corresponding wave shape are shown in Figure~\ref{fig:case2-sols}.  

To prove that this singular heteroclinic connection persists for $\varepsilon > 0$, we use Fenichel theory as well as canard theory, as in \cite{Harley_vanHeijster_Marangell_Pettet_Wechselberger_13}.  Firstly, we know that the last segment of the solution that connects to $(u_W,w_W)$ on $S_a$ persists due to Fenichel theory, as it exists solely on $S_a$ and away from the fold curve.  Fenichel theory also guarantees that the fast flow through the layer problem persists for $\varepsilon > 0$, given that a transversality condition is satisfied.  This transversality condition ensures that the two slow segments of the full problem intersect transversely; see Appendix~\ref{ap:transversality} for the computation.  

The persistence of the canard solution that leaves $(u_H,w_H)$ and passes through the folded saddle follows from results from canard theory \cite{Krupa_Szmolyan_01, Szmolyan_Wechselberger_01, Wechselberger_12}.  Therefore, the {\em singular} heteroclinic orbit of \eref{eq:reduced} perturbs smoothly to a nearby orbit of the full system \eref{eq:slow}, and the corresponding travelling wave solution in Case 2 persists as a nearby solution of \eref{eq:model-with-diff} for $0 \leq \varepsilon \ll 1$.  
\qed\end{proof}

\section{Generalised results}
\label{sec:general-results}

In the previous section, we proved the existence of a travelling wave solution to \eref{eq:model} for two sets of parameter values under some mild assumptions.  For the first parameter set, the solution was smooth and had previously been identified in \cite{Pettet_McElwain_Norbury_00}, while for the second parameter set, the solution contained a shock and was previously unrecognised as its existence relied on the interaction with a canard point.  We also proved that these solutions persist as solutions of \eref{eq:model-with-diff}, the extended model with small diffusion for both species.  

While the main results of Section~\ref{sec:gspm}, Theorem~\ref{theorem:case1} and Theorem~\ref{theorem:case2}, apply to the specific parameter regimes Case 1 and Case 2 given in \eref{eq:cases}, the qualitative results apply more broadly.  Furthermore, the methods used in Section~\ref{sec:gspm} can be applied for any choice of the parameters.  The layer problem is independent of $\alpha$ and $\beta$ and the results are not conditional on a particular value of $c$.  Consequently, the analysis of the layer problem holds for any choice of parameters.  The difference in the analysis for alternative parameter regimes arises in the reduced problem, in particular, in constructing the phase plane of the desingularised system.  In this section, we generalise the results of the previous section to a broader range of parameters.  

\subsection{Extending the results of Theorem~\ref{theorem:case1}}

Firstly, we consider the smooth travelling wave solution identified in Theorem~\ref{theorem:case1}.  The proof of Theorem~\ref{theorem:case1} does not impose any conditions on the parameters.  Thus, subsequent to Corollary~\ref{cor:case1}, we have the following result.  

\begin{corollary}\label{cor:smooth}
Assume the parameters are such that \eref{eq:ds} possesses no limit cycles and \eref{eq:poly} has no real solutions.  Then, there exists an $\varepsilon_0 > 0$ such that for $\varepsilon \in [0,\varepsilon_0]$, \eref{eq:model-with-diff} possesses a smooth travelling wave solution, connecting $(u_H,w_H)$ with $(u_W,w_W)$.  
\end{corollary}

This result holds whether $(u_H,w_H)$ is an unstable spiral or an unstable node but the shape of the travelling wave solution will be qualitatively different: if $(u_H,w_H)$ is an unstable spiral, the solution will oscillate around the healed state; if $(u_H,w_H)$ is an unstable node, while the solution may still be nonmonotone, it will not oscillate.  

The regions labelled $\emptyset$ in Figure~\ref{fig:type_pos_canards} are so-labelled because \eref{eq:poly} has no real, positive solutions.  However, the numerics suggest that in fact in these regions, \eref{eq:poly} has no real solutions, positive or negative.  Therefore, Corollary~\ref{cor:smooth} implies that in these regions, \eref{eq:model-with-diff} exhibits smooth travelling wave solutions.  

\subsection{Extending the results of Theorem~\ref{theorem:case2}}

Secondly, we consider the shock-fronted (in the singular limit) travelling wave solution identified in Theorem~\ref{theorem:case2}.  In this instance the results are not as easily generalisable.  Certainly, numerical results suggest that qualitatively different solution behaviour is observed even within the same region of Figure~\ref{fig:type-D} as the Case 2 parameter regime.  For example, for $\alpha = 0.4$, $\beta = 2.5$ and $c = 0.78$, numerical results suggest that a smooth connection between $(u_H,w_H)$ and $(u_W,w_W)$ exists, indicating the existence of a travelling wave solution qualitatively similar to the one identified in Theorem~\ref{theorem:case1}; see Figure~\ref{fig:smooth1}.  

\begin{figure}
\centering
\subfloat[$\alpha = 0.4$, $\beta = 2.5$, $c = 0.78$]{\label{fig:smooth1}\includegraphics[width=0.5\textwidth]{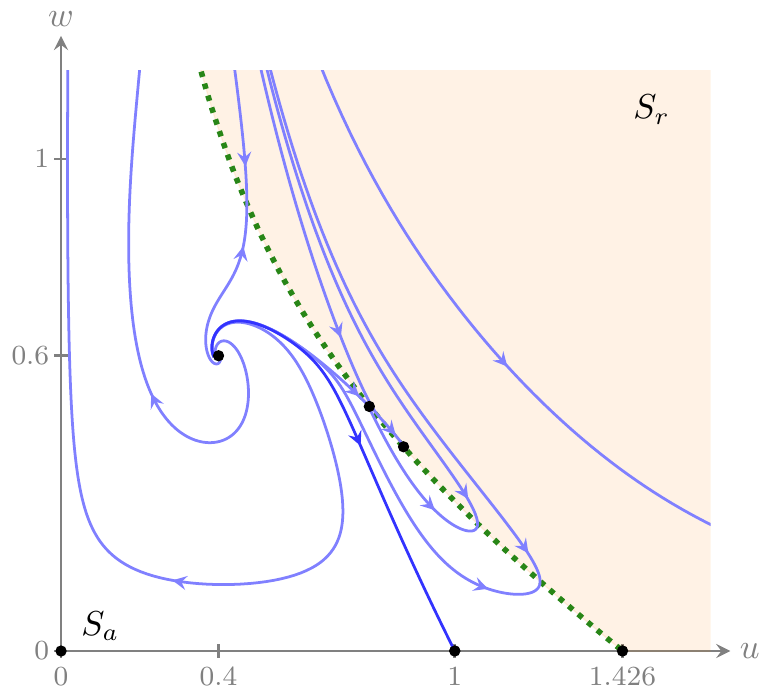}}
\subfloat[$\alpha = 0.5$, $\beta = 2.5$, $c = 0.7$]{\label{fig:shock2}\includegraphics[width=0.5\textwidth]{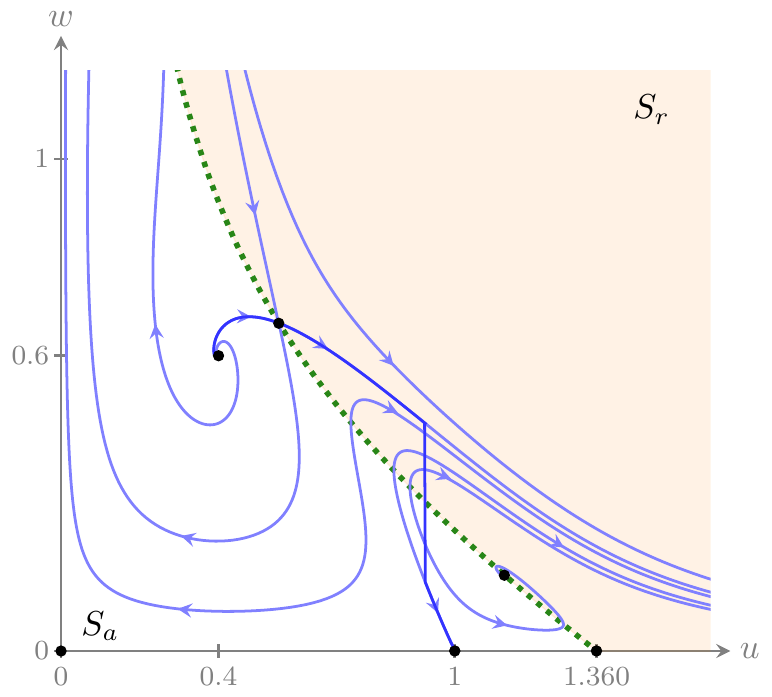}}\\
\subfloat[$\alpha = 1$, $\beta = 2.5$, $c = 1$]{\label{fig:shock3}\includegraphics[width=0.5\textwidth]{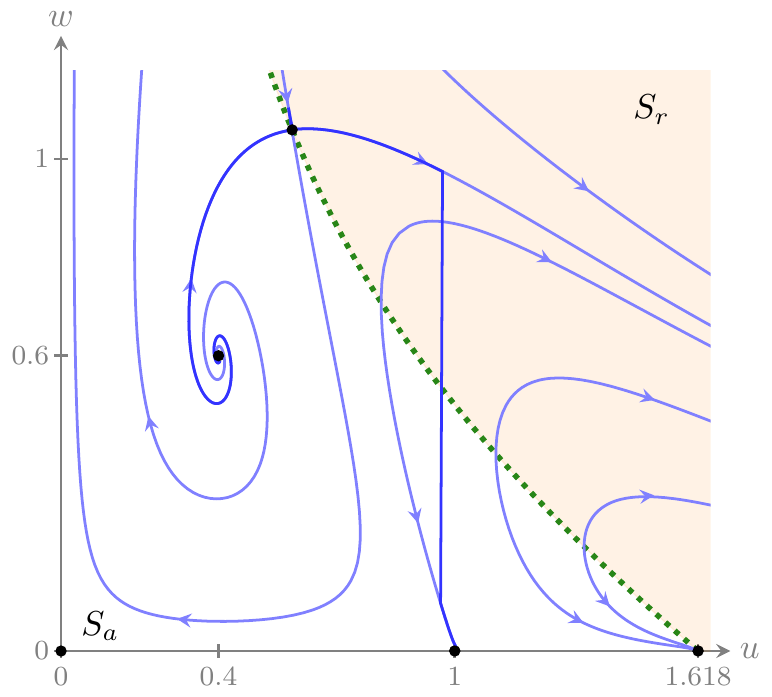}}
\subfloat[]{\label{fig:points}\includegraphics[width=0.5\textwidth]{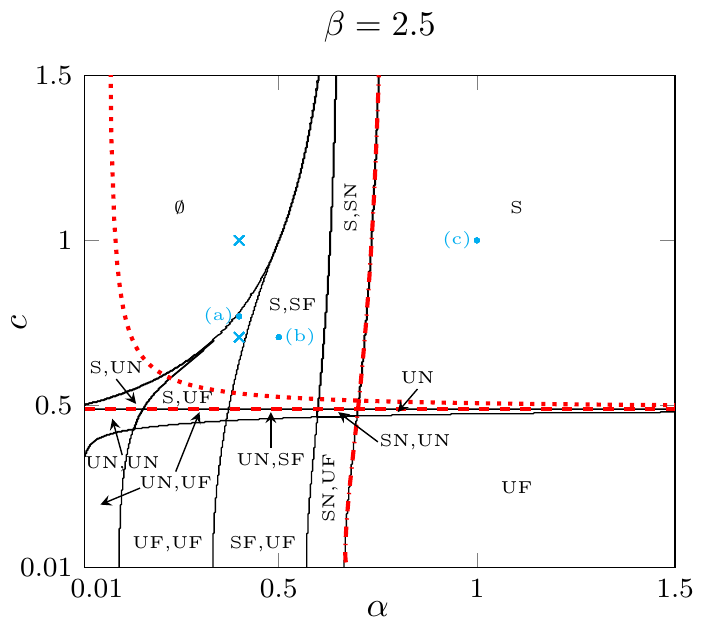}}
\caption{Illustrations of the different solution behaviours observed within the same region of parameter space as Case 2 and similar solution behaviours observed in other regions of parameter space.  The last panel indicates the locations in parameter space of the parameters used to construct the three phase planes.  The first panel depicts a smooth connection in the presence of canard points (in the same region as Case 2), the second and third panels depict two parameter regimes (in different regions to Case 2) where heteroclinic orbits containing a jump can be constructed.}
\label{fig:eg-pps}
\end{figure}

This implies that the smooth heteroclinic connection between $(u_H,w_H)$ and $(u_W,w_W)$ in the absence of canard points (corresponding to roots of \eref{eq:poly}), is not necessarily destroyed the moment a canard point appears in the positive quadrant.  Rather, the smooth connection is destroyed at a smaller (for fixed $\alpha$ and $\beta$) value of $c$: $c = c^{\ast}(\alpha,\beta)$.  

For example, consider the two parameter regimes Case 1 and Case 2, which differ only in the wavespeed $c$.  Section~\ref{sec:gspm} demonstrated that in the former regime, the stable manifold of $(u_W,w_W)$ connects (in backward $z$) to $(u_H,w_H)$, whereas in the latter regime, the stable manifold of $(u_W,w_W)$ connects (in backward $z$) to $(u_{C_1},w_{C_1})$.  Therefore, by continuity, there exists a $c = c^{\ast}(\alpha,\beta)$ at which the stable manifold of $(u_W,w_W)$ connects (in backward $z$) to $(u_{C_2},w_{C_2})$.  This is the point at which the smooth connection is destroyed.  Numerical results suggest that for the particular values of $\alpha$ and $\beta$ given in \eref{eq:cases}, $c^{\ast}(2/5,5/2) \approx 0.755 < \tilde{c}(2/5,5/2) \approx 0.785$, where we recall that $\tilde{c}(\alpha,\beta)$ is the value of $c$ at which two canard points appear due to a FSN I bifurcation.  

Another implication of the existence of smooth connections after the appearance of canard points in the positive quadrant is the possibility of non-unique solutions.  Although, as discussed above, a heteroclinic connection between $(u_W,w_W)$ and $(u_{C_2},w_{C_2})$ does not exist the instance canard points appear, a connection between $(u_H,w_H)$ and $(u_{C_2},w_{C_2})$ appears (numerically) to exist as soon as a canard point appears.  Thus, it is theoretically possible that both smooth and shock-fronted solutions exist under the same parameter regime.  However, note that for all the parameter regimes we tested numerically this was not observed.  

In contrast, there are regions other than the region of Figure~\ref{fig:type-D} where the Case 2 parameter regime lives, where solutions qualitatively similar to the one in Theorem~\ref{theorem:case2} appear to exist: for example, if $\alpha = 0.5$, $\beta = 2.5$ and $c = 0.7$, or $\alpha = 1$, $\beta = 2.5$ and $c = 1$; see Figure~\ref{fig:shock2} and Figure~\ref{fig:shock3}.  Similar shock-like solutions are theoretically possible in any parameter regime where $(u_H,w_H)$ is on $S_a$ and a folded saddle canard point is present in the first quadrant.  Note that in the parameter regimes illustrated in Figure~\ref{fig:type_pos_canards}, if a folded saddle canard in present in the first quadrant, $(u_H,w_H)$ will be on $S_a$ ($c > c_1(\beta)$), and furthermore, if $(u_H,w_H)$ lives on $S_r$ ($c < c_1(\beta)$), no folded saddles are observed.  

Therefore, assume that the parameters are such that a folded saddle is present in the positive quadrant and $(u_H,w_H)$ lives on $S_a$.  Then, to prove the existence of a travelling wave solution with similar properties to the one identified in Theorem~\ref{theorem:case2}, using the methods of Section~\ref{sec:gspm}, one must check the following.
\begin{itemize}
\item A connection exists between $(u_H,w_H)$ and the folded saddle canard $(u_{C_k},w_{C_k})$.  If the canard solution entering the folded saddle does so from below the non-trivial $w$-nullcline (the equivalent region corresponding to region 1 in Figure~\ref{fig:case2-monotonicity}), then a connection between $(u_H,w_H)$ and the relevant folded saddle canard $(u_{C_k},w_{C_k})$ exists.  Therefore, computing the stable eigenvector of the corresponding saddle of \eref{eq:ds} is sufficient to determine the existence or nonexistence of such a connection.  
\item The conditions for a jump through the layer flow are met, connecting the canard solution on $S_r$ to the stable manifold of $(u_W,w_W)$ on $S_a$.  That is, there exists a $u$ such that the $w$-coordinates along the canard solution and $\mathcal{W}^S(u_W,w_W)$ are equidistant from the fold curve.  If the connections of Conjecture~\ref{conj:hits-wall} can be shown to exist, the conditions for a jump will automatically be met.  However, note that Conjecture~\ref{conj:hits-wall} is a sufficient condition not a necessary condition.  
\item The transversality condition in Appendix~\ref{ap:transversality} is satisfied.  This is required for the proof that the singular limit solution persists for $\varepsilon > 0$.  
\end{itemize}

\subsection{Identifying other potential solutions}

Finally, we discuss other potential solutions.  As mentioned previously, the methods used in Section~\ref{sec:gspm} apply to a broad range of parameter values and so we have the ability to identify other potential heteroclinic orbits of \eref{eq:model-with-diff} (and therefore \eref{eq:model}) for different parameter values.  The existence of other solutions depends on the locations and type of canard points in the positive quadrant of the phase plane of \eref{eq:reduced}.  Note that the canard point $(u_{C_0},w_{C_0})$ exists in all scenarios.  However, since it is either a folded saddle or folded node, any trajectory passing through it cannot correspond to a physically relevant solution as doing so would result in the $w$-solution becoming negative.  Hence, we neglect it in the discussion below.  

\subsubsection{Smooth solutions}

As discussed above, when there are no canards points, only smooth connections can be made between the end states.  Smooth connections are also possible when $(u_H,w_H)$ lives on $S_a$ and any number or type of canard points are present, if the heteroclinic orbit does not pass through a canard point but connects the end states while remaining on $S_a$; one example where this occurs (numerically) is illustrated in Figure~\ref{fig:smooth1}.  

\subsubsection{Shock-like solutions}

If $(u_H,w_H)$ lives on $S_r$, shock-like solutions can exist irrespective of the number or type of canard points present; a schematic is given in Figure~\ref{fig:poss-c}.  (Due to the FSN II bifurcation at $c = c_1(\beta)$, there will always be at least one canard point when $(u_H,w_H)$ is on $S_r$; see Remark~\ref{remark:FSN-II} and Section~\ref{subsec:recovering-reduced-problem}).  A trajectory leaving $(u_H,w_H)$ simply evolves on $S_r$, until some point at which the jump condition is satisfied, upon which it jumps through the fast system and connects to $\mathcal{W}^S(u_W,w_W)$.  Proving the existence of a connection of this kind is similar to the proof of Theorem~\ref{theorem:case2} but without the complication of the canard point.  Rather, one only needs to show that the conditions for a jump connection are satisfied, in this instance, between the unstable manifold of $(u_H,w_H)$ and the stable manifold of $(u_W,w_W)$.  A trajectory of this kind is numerically computed in \cite{Marchant_Norbury_02} with $\alpha = 7/10$, $\beta = 10/7$ and $c \approx 0.24$, where the jump through the fast system connects directly to $(u_W,w_W)$ such that the corresponding travelling wave has semi-compact support in $w$.  

\subsubsection{Solutions involving folded nodes}

Thus far, the only solutions considered involving a canard point are those where the canard point in question is a folded saddle.  As previously discussed, folded foci do not allow trajectories to pass through them \cite{Wechselberger_12} and so do not give rise to new heteroclinic connections.  However, potential new solutions arise if we consider parameter regimes where folded nodes exist.  

Unlike folded saddles, which admit a single trajectory passing from $S_a$ to $S_r$ along the canard solution (and similarly from $S_r$ to $S_a$ along the faux-canard solution), folded nodes admit multiple trajectores but only in one direction \cite{Krupa_Szmolyan_01, Szmolyan_Wechselberger_01, Wechselberger_05, Wechselberger_12}.  If the corresponding node of \eref{eq:ds} is stable, the folded node of \eref{eq:pre-ds} will admit a {\em wedge} of trajectories passing from $S_r$ to $S_a$.  Alternatively, if the corresponding node of \eref{eq:ds} is unstable, trajectories of \eref{eq:pre-ds} only pass from $S_r$ to $S_a$ through the folded node.  

Proving the existence of solutions involving a folded node is more complicated than when only a folded saddle is involved and is beyond the scope of this article.  While in the singular limit one might be able to construct a heteroclinic orbit, proving the persistence for $\varepsilon > 0$ is more challenging as only finitely many canards persist for $\varepsilon > 0$ out of the continuum of singular (for $\varepsilon = 0$) canards.  The following is a list of possible singular heteroclinic orbits involving folded nodes.  
\begin{itemize}
\item If $(u_H,w_H)$ lives on $S_a$ and a folded node and a folded saddle are present, a smooth connection between the end states is possible, involving both canard points.  This connection would pass onto $S_r$ through one of the canard points and then back to $S_a$ through the other.  See Figure~\ref{fig:poss-a} for a schematic, where the two canards must correspond to equilibria of \eref{eq:ds} that are either a stable node and a saddle or a saddle and an unstable node, in order of crossing.  (This type of connection is also possible if two folded saddle canard points are present.  However, we did not identify any regions of parameter space where this occured; see Figure~\ref{fig:type_pos_canards}.)
\item If $(u_H,w_H)$ lives on $S_a$ and a folded node corresponding to a stable node of \eref{eq:ds} is present (alone or with another canard point), jump solutions are possible.  These solutions (in the singular limit) may be similar in appearance to the jump solutions involving a folded saddle such as the one identified in Theorem~\ref{theorem:case2}.  
\item If $(u_H,w_H)$ lives on $S_r$ and a folded node corresponding to a unstable node of \eref{eq:ds} is present (alone or with another canard point), a smooth connection is possible.  This connection would simply connect $(u_H,w_H)$ on $S_r$ to $(u_W,w_W)$ on $S_a$, via the folded node; see Figure~\ref{fig:poss-d} for a schematic.  (This type of connection is also possible if the canard point is a folded saddle.  However, with $(u_H,w_H)$ on $S_r$, we did not observe any folded saddles in the positive quadrant; see Figure~\ref{fig:type_pos_canards}.)
\end{itemize}

\begin{remark}
We emphasise that we are merely suggesting that these scenarios are possible in theory, potentially yielding an even broader class of solutions to \eref{eq:model-with-diff} (and hence \eref{eq:model} and \eref{eq:unscaled-model}); there is no guarantee they will be observed in practice.  Proving the existence (or non-existence) of these solutions requires further analysis of the vector field for each parameter regime in the singular limit and of the persistence of the canard solutions for $0 < \varepsilon \ll 1$.  This is not done in this article but rather left for future research.  We postulate that this may pose a considerable challenge in the general case and instead may have to be investigated on a case by case basis, as is done in this article.  
\end{remark}

\begin{figure}[ht]
\centering
\subfloat[]{\label{fig:poss-c}\includegraphics[width=0.32\textwidth]{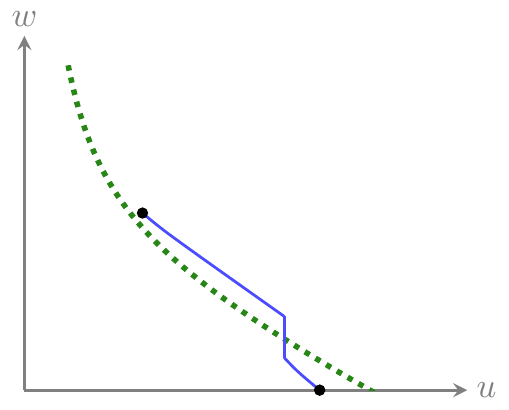}}
\subfloat[]{\label{fig:poss-a}\includegraphics[width=0.32\textwidth]{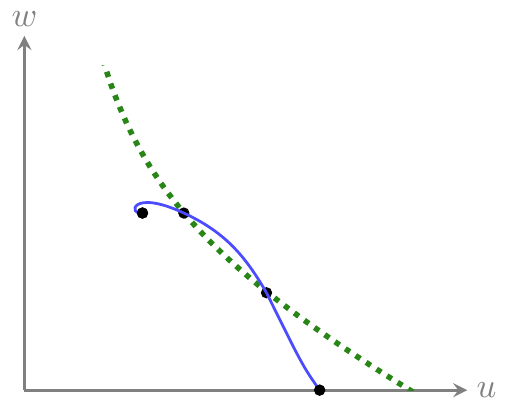}}
\subfloat[]{\label{fig:poss-d}\includegraphics[width=0.32\textwidth]{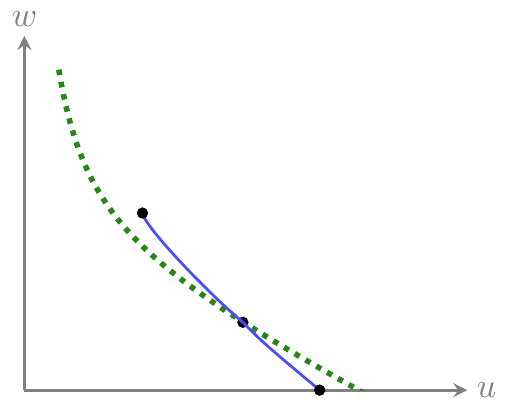}}
\caption{Schematics of other possible solution trajectories.}
\label{fig:possibles}
\end{figure}

\section{Conclusion}

In this article, we used GSPT and canard theory to prove the existence of travelling wave solutions to \eref{eq:model-with-diff}.  From the outset, the purpose was two-fold: firstly, to rigourously prove the existence of travelling wave solutions to \eref{eq:model-with-diff} (and therefore \eref{eq:model}) for the two parameter regimes considered in \cite{Pettet_McElwain_Norbury_00}, given in \eref{eq:cases}; and, secondly, to generalise these results and to infer the types of solutions that may be observed for a broader range of parameters.  The former comprised Section~\ref{sec:gspm}, the latter Section~\ref{sec:general-results}.  

One of the original motivations for considering the system \eref{eq:unscaled-model} and \eref{eq:model} is that it models the fundamental dynamics associated with one aspect of epidermal wound healing, especially as it relates to the speed of wound closure rather than the detailed cellular architecture that develops.  Characterising the behaviour of travelling wave solutions to this model, in particular by the relationship between the kinetic rate parameter $\alpha$ and the threshold value $1/\beta$ (that separates angiogenic extension from retraction), provides some insight into targets for wound healing interventions that have a likely or significant impact on healing speed.  For example, the parameter $\alpha$ may be thought of as representing the sensitivity of the process of vessel cell proliferation to the growth factor MDGF.  Further, it may be possible to infer a potential wound healing diagnostic from the relationship between wavespeed and sharpness of the invading angiogenic front discussed at length here.  The observed presence of slow moving sharp-fronted angiogenic fronts in a wound may indeed indicate a compromised proliferative or chemotactic response to a specific growth factor, such as MDGF.  

Of course, the observability of the travelling wave solutions depends not only their existence but also their stability.  Determining the stability of the travelling wave solutions is beyond the scope of this article.  However, from both a mathematical and biological perspective, it is an important aspect of the analysis and, accordingly, the topic of future research.  

\subsection*{Acknowledgements}

This research was supported under the Australian Research Council's Discovery Projects funding scheme (project number DP110102775).

\appendix

\section{Proof of Lemma~\ref{lemma:S}}
\label{ap:layer}

\begin{proof}
The stability of $S$ is determined by examining the eigenvalues of the Jacobian of the layer problem,
\[ J_L = \left[\begin{array}{ccc} -c & 0 & 0 \\ -f(2u,w) & -c & u \\ 0 & w & -c + v \end{array}\right], \]
where $u$, $v$ and $w$ are restricted to $S$.  The eigenvalues 
\[ \lambda_1 = -c \quad {\rm and} \quad \lambda_2 = -c + \frac{1}{2}\left(v - \sqrt{v^2 + 4uw}\right) \]
are negative for all $u > 0$, $w > 0$, while the third eigenvalue,
\[ \lambda_3 = -c + \frac{1}{2}\left(v + \sqrt{v^2 + 4uw}\right), \]
can change sign.  Consequently, the layer problem exhibits a saddle-node (fold) bifurcation along $\lambda_3 = 0$.  This implies that $S$ is folded with the fold curve corresponding to $\lambda_3 = 0$ (which coincides with the wall of singularities $F(u)$, defined in \eref{eq:wall}), provided the following non-degeneracy and transversality conditions are satisfied \cite{Wechselberger_12}.  

Firstly, the non-degeneracy condition is
\[ \bi{p} \cdot (D^2_{\bi{U}\bi{U}}\bi{G})(\bi{U},\hat{\bi{U}})(\bi{q},\bi{q}) \neq 0, \]
or equivalently, 
\[ \bi{p} \cdot \bi{B}(\bi{q},\bi{q}) \neq 0, \]
where, 
\[ \bi{B}_i(\bi{q},\bi{q}) = \sum_{j,k = 1}^3{\frac{\partial^2\bi{G}_i}{\partial\bi{U}_k\partial\bi{U}_j}}\bi{q}_j\bi{q}_k. \]
Here, $\bi{U} = (u,v,w)$, $\hat{\bi{U}} = (\hat{u},\hat{w})$ and $\bi{G} = (u_y,v_y,w_y)$, with $u_y$, $v_y$ and $w_y$ defined in \eref{eq:layer}.  Moreover, $\bi{U}$, $\hat{\bi{U}}$ and all the derivatives of $\bi{B}(\bi{q},\bi{q})$ are evaluated along $\lambda_3 = 0$.  For example, we have $\bi{U} = \left(u,\dfrac{c^2 - u + u^2}{2c},F(u)\right)$.  

The vectors $\bi{p}$ and $\bi{q}$ are the adjoint null-vector and null-vector of $J_L$, respectively, normalised via $\bi{q} \cdot \bi{q} = \bi{p} \cdot \bi{q} = 1$.  An easy computation shows 
\[ \bi{p} = \frac{1}{P}\left( F(u)\frac{c^2 - u + 3u^2}{2cu}, F(u), c \right), \quad \bi{q} = \frac{1}{Q}\left( 0, u, c \right)^T, \]
where 
\[ P = \frac{3c^2 + u - u^2}{2Q}, \quad Q = \sqrt{c^2 + u^2}. \]
Therefore, we have
\[ \bi{p} \cdot \bi{B}(\bi{q},\bi{q}) = \frac{2c^2u}{PQ^2} \neq 0 \quad {\rm for} \quad c,u \neq 0. \]

Secondly, the transversality condition is
\[ \bi{p} \cdot (D_{\hat{\bi{U}}}\bi{G})(\bi{U},\hat{\bi{U}}) \neq \textbf{0}, \]
where once again, the vectors are evaluated along $\lambda_3 = 0$.  In our case, this becomes
\[ \bi{p} \cdot (D_{\hat{\bi{U}}}\bi{G})(\bi{U},\hat{\bi{U}}) = \bi{p} \cdot \left[\begin{array}{ccc} 1 & 0 \\ 0 & 0 \\ 0 & 1 \end{array} \right] = \frac{1}{P}\left( F(u)\frac{c^2 - u + 3u^2}{2cu}, c\right) \neq \textbf{0}.\] 
\qed\end{proof}

\section{Transversality condition}
\label{ap:transversality}

\begin{lemma}\label{lemma:transversality}
$\mathcal{W}^U(u_{C_2},w_{C_2})$ on $S_r$ and $\mathcal{W}^S(u_W,w_W)$ intersect transversely near the jump location.  
\end{lemma}

\begin{proof}
Since jumps through the layer flow are equidistant from the fold curve in $w$, we can express the value of $w$ where the jump lands on $S_a$ in terms of its value where the jump leaves $S_r$ and its value along the fold curve:
\[ w_{\rm landing}(u) = 2F(u) - w^{\ast}(u), \]
with $w^{\ast}$ the value of $w$ when the trajectory leaves $S_r$.  Therefore, to ensure a transverse intersection, we require that
\[ 2\diff{F(u)}{u} - \left.\diff{w}{u}\right|_{w^{\ast}} \neq \left.\diff{w}{u}\right|_{2F(u) - w^{\ast}}. \]
By evaluating the above expressions, we find that
\[ 2\diff{F(u)}{u} - \left.\diff{w}{u}\right|_{w^{\ast}} - \left.\diff{w}{u}\right|_{2F(u) - w^{\ast}} = \frac{\alpha c^2 (\beta u - 1)(u - 1)}{u f(u,w^{\ast})(c^2 - uw^{\ast})}, \]
which is non-zero provided $\alpha,c \neq 0$ and $u \neq 1/\beta$, $u \neq 1$, as is the case here.  
\qed\end{proof}

\bibliographystyle{unsrt}
\bibliography{papers}

\end{document}